\documentclass[11pt,a4paper,reqno]{amsart}

\usepackage{mathtools,amssymb}
\usepackage{tikz-cd}
\usepackage{enumitem} \setlist{itemsep=1mm} 
\usepackage[hidelinks]{hyperref}
\usepackage{url}
\usepackage[textwidth=27mm,textsize=scriptsize, color=green!20]{todonotes}
\newcommand{\change}[1]{{#1}}

\newtheorem{theorem}{Theorem}[section]
\newtheorem{proposition}[theorem]{Proposition}
\newtheorem{lemma}[theorem]{Lemma}
\newtheorem{corollary}[theorem]{Corollary}
\newtheorem{conjecture}[theorem]{Conjecture}

\theoremstyle{definition}
\newtheorem{definition}[theorem]{Definition}
\newtheorem{remark}[theorem]{Remark}
\newtheorem{example}[theorem]{Example}

\numberwithin{equation}{section}

\DeclareMathOperator{\diag}{diag}
\DeclareMathOperator{\NP}{NP}

\DeclareMathOperator{\PU}{PU}
\DeclareMathOperator{\Hom}{Hom}
\DeclareMathOperator{\supp}{supp}

\DeclareMathOperator{\conv}{conv}
\DeclareMathOperator{\Ker}{Ker}

\DeclareMathOperator{\vol}{vol}
\DeclareMathOperator{\rank}{rank}
\DeclareMathOperator{\init}{init}
\DeclareMathOperator{\Aut}{Aut}

\renewcommand{\div}{\operatorname{div}}

\newcommand{\FS}{\mathrm{FS}}

\renewcommand{\and}{\quad \text{and} \quad}

\begin{document}

\title[K\"ahler--Einstein toric submanifolds]{K\"ahler--Einstein toric
  submanifolds of the projective space}

\author[Di Scala]{Antonio J. Di Scala}
\address{\hspace*{-6.3mm}
  Dipartimento di Scienze Matematiche, Politecnico di Torino, 10129 Torino, Italy \vspace*{-2.8mm}}
\address{\hspace*{-6.3mm} {\it Email address:} {\tt antonio.discala@polito.it}}

\author[Sombra]{Mart\'in~Sombra}
\address{\hspace*{-6.3mm} Institució Catalana de Recerca i Estudis Avançats, 08010 Barcelona, Spain \vspace*{-2.8mm}}
\address{\hspace*{-6.3mm}  Departament de Matem\`atiques i
  Inform\`atica, Universitat de Barcelona,  08007
  Bar\-ce\-lo\-na, Spain \vspace*{-2.8mm}}
\address{\hspace*{-6.3mm} Centre de Recerca Matem\`atica, 08193 Bellaterra, Spain \vspace*{-2.8mm}}
\address{\hspace*{-6.3mm} {\it Email address:} {\tt martin.sombra@icrea.cat}}

\date{\today} 

\begin{abstract}
  We show that the K\"ahler--Einstein metrics on the four families of
  symmetric toric Fano manifolds due to Batyrev and Selivanova cannot
  be induced by immersions into projective spaces equipped with the
  Fubini--Study metric.  We also show a similar conclusion for the
  non-symmetric examples discovered by Nill and Paffenholz. As an
  application we obtain that a centrally symmetric toric Fano manifold
  admits a K\"ahler--Einstein metric induced by a projective immersion
  if and only if it is a product of projective lines.  These results
  confirm for each of these examples a broader conjecture
  characterizing the K\"ahler--Einstein metrics that can be induced by
  projective immersions.
\end{abstract}

\maketitle

\thispagestyle{empty}

\vspace{-6mm}

\section{Introduction}
\label{sec:introduction}

It is well-known that the Fubini--Study metric $g_{\FS}$ on a complex
projective space $\mathbb{P}^s$ is Einstein. Less known are the
compact complex manifolds $X$ having an immersion
$\varphi\colon X\to \mathbb{P}^s$ such that the pullback K\"ahler
metric $g =\varphi^{*}g_{\FS}$ is Einstein.
A result of Hulin~\cite{Hulin:kempe} shows that in this case the Ricci curvature is a positive multiple of the metric, which implies that $X$ is
a Fano manifold.

After the proof of the Yau--Tian--Donaldson conjecture, the existence of
a K\"ahler--Einstein metric $g$ on a Fano manifold~$X$ can be
characterized by the algebro-geometric notion of
K-poly\-sta\-bil\-i\-ty~\cite{CDS:kemfm,Tian:KsKEm}. When this
condition is verified, it is natural to ask if the metric $g$ is
\emph{projectively induced} in the sense that there exists an
immersion $\varphi \colon X \to \mathbb{P}^s$ such that
\begin{equation}
  \label{eq:36}
g = \varphi^* g_{\FS}. 
\end{equation}

In \cite{Calabi} Calabi showed that for an arbitrary K\"ahler manifold
$(X,g)$ there are strong restrictions for the existence of a
projective immersion inducing the metric, and so one expects
\eqref{eq:36} to occur only in special situations. This contrasts
sharply with the analogue in Riemannian geometry, since by Nash's
theorem every Riemannian manifold can be isometrically embedded into a
real Euclidean space~\cite{Nash}.

The problem is therefore to classify the K\"ahler--Einstein compact
manifolds that satisfy the condition \eqref{eq:36}. In this direction
Calabi showed that a K\"ahler--Einstein metric $g$ on a curve~$X$
satisfies this condition if and only if~$X$ is the projective line and
$g$ is a positive integer multiple of the Fubini--Study metric
\cite{Calabi}. Continuing with the work of Smyth~\cite{S:dgch} and
Chern~\cite{Chern:ehkmchc} on hypersurfaces, all known examples of
K\"ahler--Einstein submanifolds of a projective space turned out to be
homogeneous \cite{Hano:ecicps,Nakagawa:ekmicps,
  Matsuyama:dekscsf,Tsukada:EKsccsf}.  This has motivated the
following folklore conjecture.

\begin{conjecture}
  \label{conj:1}
  Let $g$ be a projectively induced K\"ahler--Einstein metric on a
  compact manifold~$X$. Then $(X,g)$ is a homogeneous K\"ahler
  manifold.
\end{conjecture}

Ogiue's problem about irreducible or Einstein submanifolds of the
projective space with parallel second fundamental form \cite[Problem
7]{Ogiue:dgkm} can be regarded as a predecessor of this statement. In its current form, the
conjecture  appears for instance in Tosatti's
Zentralblatt review of \cite{HWZ:hiinkemidcps} and in~\cite[Conjecture
4.2]{LZ:kikmcsf}.  The analogous problem of classifying the
K\"ahler--Einstein submanifolds of the complex Euclidean and hyperbolic
spaces was also proposed by Ogiue \cite[Problem 8]{Ogiue:dgkm} and
solved by Umehara~\cite{MR902577}.
Indeed the homogeneous K\"ahler--Einstein submanifolds of a projective space
have been determined by Takeuchi~\cite{Takeuchi:hkscps}, and so a
positive answer to Conjecture \ref{conj:1} would complete the
classification of the  K\"ahler--Einstein submanifolds of the complex space
forms.

\bigskip

The study of K\"ahler--Einstein metrics on compact toric manifolds has
experienced an important development in recent years, see for instance
\cite{BB:rmaekrstlfv} and the references therein.  Their
classification in dimensions up to $4$ was obtained by Batyrev and
Selivanova \cite{BS:ekmstfm} and then in dimensions up to $7$
combining the classification of unimodular Fano polytopes by Obro
\cite{Obro07} with the combinatorial characterization for the
existence of a K\"ahler--Einstein metric on the associated Fano
manifolds by Wang and Zhu \cite{WZ:krstmpfcc}.

Using these results Arezzo, Loi and Zuddas \cite{ArezzoLoiZuddas_hbm}
and Manno and Salis~\cite{MS:TKEmicps} proved that a K\"ahler--Einstein
toric manifold of dimension at most $6$ is not projectively induced
unless it is a product of projective spaces, under the additional
hypothesis that the immersion is given by a complete linear series.
Unfortunately a complete classification of K\"ahler--Einstein compact
toric manifolds is not available yet, and so this approach to the
toric case of the conjecture cannot be extended to the higher
dimensional situation.

In this paper we develop algebraic tools to handle such higher
dimensional situations.  To explain our results, for an integer
$n\ge 1$ let $X$ be an $n$-dimensional toric Fano manifold and
$\Sigma$ its associated fan on a vector space
$N_{\mathbb{R}}\simeq \mathbb{R}^{n}$.  Following~\cite{VK:tfvrs,
  BS:ekmstfm} this toric Fano manifold is said to be \emph{centrally
  symmetric} if the reflection on $N_{\mathbb{R}}$ with respect to the
origin is an automorphism of~$\Sigma$, and \emph{symmetric} if the
group of automorphisms of $\Sigma$ fixes only the origin of
$N_{\mathbb{R}}$. Clearly a centrally symmetric toric Fano manifold is
symmetric, but the converse does not always hold
\cite{jin2025symmetrynotionstoricfanos}.

In \cite{BS:ekmstfm} Batyrev and Selivanova showed that every
symmetric toric Fano manifold admits a K\"ahler--Einstein metric, and
they presented four families of such toric manifolds including many of
the previously known examples.  These are the del Pezzo toric
manifolds $V_{k}$ of Voskresenskij and Klyachko~\cite{VK:tfvrs}, and
the symmetric toric Fano manifolds $S_{m,k}$, $X_{m,k}$ and $W_{m}$
respectively introduced by Sakane \cite{Sakane},
Nakagawa~\cite{Nakagawa} and themselves, see \cite[Section
4]{BS:ekmstfm} or Section~\ref{sec:large-dimensions} for details.

Here is our main result, which is a particular case of the more
technical Theorem~\ref{thm:3}.

\begin{theorem}
  \label{thm:1}
  Let $X$ be a symmetric toric Fano manifold of type $V_{k}$,
  $S_{m,k}$, $X_{m,k}$ or~$W_{m}$, and $g$ a K\"ahler--Einstein
  metric on it. Then $g$ is not projectively induced.
\end{theorem}

By a result of Voskresenskij and Klyachko, every centrally symmetric
toric Fano manifold is a product of projective lines and del Pezzo
toric manifolds \cite{VK:tfvrs} . Applying this together with Theorem
\ref{thm:3} we obtain the next consequence.

\begin{corollary}
  \label{cor:4}
  Let $X$ be a centrally symmetric toric Fano manifold. Then $X$
  admits a projectively induced K\"ahler--Einstein metric if and only
  if it is a product of projective~lines.
\end{corollary}

In \cite{NP:eketfmansrp} Nill and Paffenholz proved that there exist
K\"ahler--Einstein toric Fano manifolds that are not symmetric by
exhibiting two examples in dimensions~$7$ and~$8$. As another
application of our techniques we show that these metrics are neither
projectively induced (Theorem \ref{thm:5}), thus providing further
evidence for Conjecture~\ref{conj:1}.

\bigskip We next sketch our strategy. We work with the equivalent
point of view of forms instead of that of metrics, and so we consider
a pair $(X,\omega)$ consisting of an $n$-dimensional toric Fano
manifold and a K\"ahler--Einstein form induced by an immersion
$\varphi\colon X\to \mathbb{P}^{s}$ from the Fubini--Study form of this
projective space.

Let $M \simeq \mathbb{Z}^{n}$ be the group of characters of the
complex torus acting on $X$. Applying Calabi's rigidity theorem and
the basic properties of critical metrics we reduce to the situation
where $\varphi$ is equivariant with respect to the usual toric
structure on $\mathbb{P}^{s}$ (Lemmas \ref{lem:1} and
\ref{lem:15}). In this situation the immersion is represented by a
family of monomials that we can use to construct a Laurent polynomial
\begin{math} 
  p_{\varphi}\in \mathbb{R}_{>0}[M].
\end{math}

We then define an algebraic differential operator
\begin{math}
\mu \colon \mathbb{C}[M]\rightarrow \mathbb{C}[M]  
\end{math}
allowing to translate the Einstein condition for $\omega$ into the
equation (Proposition \ref{prop:1})
\begin{displaymath}
  \mu(p_{\varphi})=p_{\varphi}^{n},
\end{displaymath}
similar to that considered in
\cite{ArezzoLoiZuddas_hbm,MS:TKEmicps}. To study it we consider the
more flexible requirement that
\begin{math}
  \mu(p_{\varphi})
\end{math}
divides $ p_{\varphi}^{\kappa}$ for an integer $\kappa> 0$, which we
call the \emph{generalized Einstein condition (GEC)}.  This condition
is useful because it is inherited by initial parts of
polynomials. Hence the existence of a single obstructing invariant
subvariety of $X$ rules out a projectively induced Kähler--Einstein
metric on $X$.

Our key technical results are an explicit description of the Newton
polytope of~$\mu(p_{\varphi})$ and a factorization formula for the
initial parts of this Monge--Ampère polynomial~(Theorems~\ref{thm:4}
and \ref{thm:2}). The latter can be interpreted as an algebraic
version of the adjunction formula on $X$ for an invariant
hypersurface, and implies that GEC is hereditary on the invariant
submanifolds of $X$ (Proposition \ref{prop:5}).  We then study some
specific toric surfaces to show that GEC cannot be satisfied on them
(Corollaries \ref{cor:2} and~\ref{cor:3}). Since all the considered
higher dimensional toric Fano manifolds contain at least one of these
toric surfaces as an invariant submanifold, we deduce that GEC can be
neither satisfied on them. In particular, none of them admits a
projectively induced K\"ahler--Einstein metric~(Theorems~\ref{thm:3}
and~\ref{thm:5}).

In view of this approach, it is natural to ask if for every toric Fano
manifold $X$ that admits a K\"ahler--Einstein metric and is not a
product of projective spaces we have that $X$ contains one of these
specific toric surfaces as an invariant submanifold. Whereas we do not
have a conceptual argument to support this claim, we could verify it
for every example we considered, thus confirming Conjecture \ref{conj:1} for
each of them.

\bigskip

The paper is organized as
follows. Section \ref{sec:preliminaries} contains the preliminary
notions and facts from differential and toric geometries.  In Section
\ref{sec:algebraic-potentials} we give a combinatorial formula for the
polynomial Monge--Ampère operator and study its
properties. Section~\ref{sec:applications} contains the proof of our
main results.

\section{Preliminaries}
\label{sec:preliminaries}

In this section we recall the basic constructions and properties of
complex toric manifolds and start to study the toric K\"ahler forms
that are induced by projective immersions. We show that each of these
K\"ahler forms can be encoded by a Laurent polynomial, which allows to
translate the corresponding Einstein condition into an algebraic
equation.

\subsection{Toric manifolds}
\label{sec:toric-manifolds}

For an integer $n\ge1$ we denote by $\mathbb{T}$ an $n$-dimensional
(complex) torus, namely a complex group isomorphic to
$(\mathbb{C}^{\times})^{n}$, and by $\mathbb{S}\simeq (S^{1})^{n}$ its
maximal compact subtorus.

\begin{definition}[{\cite[Section~2]{Donaldson:kgtmsomlsg}}]
  \label{def:7}
  A (complex) manifold~$X$ is \emph{toric} if it contains~$\mathbb{T}$
  as a dense open subset and is equipped with an action of
  $\mathbb{T}$ extending the action of this torus on itself by
  translations.  A K\"ahler form $\omega$ on $X$ is \emph{toric} if it
  is invariant under the action of $\mathbb{S}$, in which case the
  pair $(X,\omega)$ is called a \emph{toric} K\"ahler manifold.
\end{definition}

Recall that a K\"ahler manifold $(X,\omega)$ is
\emph{Einstein} if its Ricci form $\rho$ verifies
\begin{displaymath}
  \rho=\lambda\, \omega \quad \text{ with } \lambda\in \mathbb{R}. 
\end{displaymath}

\begin{example}
  \label{exm:1}
  The $n$-dimensional (complex) projective space $\mathbb{P}^{n}$ is a
  toric complex manifold for the torus $(\mathbb{C}^{\times})^{n}$
  with the action defined for
  $ t=(t_{1},\dots, t_{n})\in (\mathbb{C}^{\times})^{n} $ and
  $z=(z_{0}:z_{1}:\cdots:z_{n})\in \mathbb{P}^{n}$ as
\begin{displaymath}
  t \cdot z =(z_{0}:t_{1} z_{1}: \cdots : t_{n} z_{n}).
\end{displaymath}
The \emph{Fubini--Study form} of $\mathbb{P}^{n}$ is the K\"ahler form
$\omega_{\FS}$ defined at each point as the pullback of the $2$-form
on $\mathbb{C}^{n+1}\setminus \{0\}$
  \begin{displaymath}
    i \, \partial\overline{\partial} \log\Big( \sum_{j=0}^{n} |z_{j}|^{2}\Big)
  \end{displaymath}
  by any local section of the projection
  $\mathbb{C}^{n+1}\setminus \{0\}\to \mathbb{P}^{n}$. This K\"ahler
  form is invariant under the action of the projective unitary
  group~$\PU(n+1)$ and is Einstein with constant $\lambda=n+1$
  \cite[Chapter~13]{Moroianu:lkg}. In particular
  $(\mathbb{P}^{n},\omega_{\FS})$ is a toric K\"ahler--Einstein
  manifold.
\end{example}

In the sequel we restrict to the case in which $X$ is a compact smooth
toric variety, that is a compact toric manifold which is also an
algebraic variety.  We present the necessary notions and facts,
referring to \cite{Fulton_toric, CoxLittleSchenck:tv} for the proofs
and more details.

Set
\begin{displaymath}
  M=\Hom(\mathbb{T}, \mathbb{C}^{\times}) \and N=\Hom(\mathbb{C}^{\times},\mathbb{T})
\end{displaymath}
for the lattices of characters and co-characters of the torus
$\mathbb{T}$. Both are isomorphic to~$\mathbb{Z}^{n}$ and dual of each
other, that is $N=M^{\vee}$ and $M=N^{\vee}$.  Set then
$N_{\mathbb{R}}=N\otimes_{\mathbb{Z}}\mathbb{R}$ and
$M_{\mathbb{R}}=M\otimes_{\mathbb{Z}}\mathbb{R}$. These vector spaces
are also dual of each other, and for $u\in N_{\mathbb{R}}$ and
$x\in M_{\mathbb{R}}$ we denote their pairing by
$\langle u,x \rangle$. Let $\mathbb{C}[M]$ be the group algebra of~$M$
over $\mathbb{C}$, and for each $m\in M$ we set
$\chi^{m}\in \mathbb{C}[M]$ for the corresponding monomial.

To a compact smooth toric variety $X$ corresponds a fan
$\Sigma=\Sigma_{X}$ on $N_{\mathbb{R}}$ that is \emph{complete} and
\emph{unimodular}, that is \change{a fan that covers the whole of this
  vector space and whose $n$-dimensional cones are generated by bases
  of the lattice $N$}. For each integer~$0\le r\le n$ we denote by
$\Sigma^{r}$ the collection of cones of $\Sigma$ of dimension $r$.

There is an inclusion-reversing bijection $\sigma\mapsto X_{\sigma}$
between the cones of $\Sigma$ and the affine toric varieties that glue
up to build $X$. The origin $0\in \Sigma$ corresponds to the
\emph{principal open subset} $X_{0}\subset X$ and is canonically
identified with the torus. The fact that $X$ is smooth implies that
for each $\sigma\in \Sigma^{r}$ we have
\begin{displaymath}
  X_{\sigma}\simeq \mathbb{C}^{r}\times (\mathbb{C}^{\times})^{n-r}.
\end{displaymath}
There is also a dimension-reversing bijection
\begin{math}
\sigma\mapsto V(\sigma)   
\end{math}
between the cones of $\Sigma$ and the invariant subvarieties of
$X$. In particular to each $1$-dimensional cone (or \emph{ray})
it corresponds an invariant hypersurface.

Every line bundle $L$ on $X$ can be realized as the line bundle
associated to a toric divisor, that is
\begin{displaymath}
  L=\mathcal{O}_{X}\Big(\sum_{\tau} a_{\tau} V(\tau)\Big),
\end{displaymath}
the sum being over the rays $\tau$ of the fan and where
$a_{\tau}\in \mathbb{Z}$ for each $\tau$. This representation is
unique up to a principal divisor of the form $\div(\chi^{m})$ with
$m\in M$. Then we associate to $L$ the lattice polytope in the vector
space $M_{\mathbb{R}}$ defined as
\begin{displaymath}
  \Delta_{L}=\{ x\in M_{\mathbb{R}} \mid \langle u_{\tau},x\rangle \ge -a_{\tau} \text{ for all } \tau\in \Sigma^{1} \}, 
\end{displaymath}
where $u_{\tau}$ denotes the smallest nonzero lattice vector in the
ray $\tau$. It is well-defined up to a translation by an element of
$M$.

When $L$ is ample the polytope $\Delta_{L}$ is $n$-dimensional and its
normal fan coincides with $\Sigma$. In particular there are
dimension-reversing bijections
\begin{displaymath}
F\longmapsto \sigma_{F}  \and \sigma\longmapsto F_{\sigma}
\end{displaymath}
between the faces of $\Delta_{L}$ and the cones of $\Sigma$.

The anticanonical line bundle of $X$ can be represented as
\begin{math}
  -K_{X} =\mathcal{O}_{X}(\sum_{\tau}  V(\tau))
\end{math}
and so its lattice polytope can be fixed to
\begin{equation}
  \label{eq:4}
  \Delta_{-K_{X}}=\{x\in M_{\mathbb{R}} \mid \langle u_{\tau},x\rangle \ge -1 \text{ for all } \tau\in \Sigma^{1} \}.
\end{equation}
If $X$ is Fano then $-K_{X}$ is ample and so $\Delta_{-K_{X}}$ is
$n$-dimensional. Hence this is a reflexive polytope, and in particular
the origin is its unique interior lattice point.

\subsection{Projectively induced forms}
\label{sec:proj-induc-kahl}

Here we establish some general facts about the K\"ahler--Einstein forms
on toric manifolds that are induced by projective immersions.

\begin{definition}
  \label{def:10}
A K\"ahler form $\omega$ on a compact manifold $X$  is  \emph{projectively induced} if there exists an
immersion \begin{math} \varphi \colon X \rightarrow \mathbb{P}^{s}
\end{math}
such that
\begin{displaymath}
  \omega=\varphi^{*} \omega_{\FS}.
\end{displaymath}
Up to reducing to a linear subspace of $\mathbb{P}^{s}$ we can assume
that the immersion is \emph{full}, that is its image is not contained
in any proper linear subspace.
\end{definition}

\begin{remark}
  \label{rem:6}
  If $(X,\omega)$ is a projectively induced compact K\"ahler manifold
  then the line bundle $L=\varphi^{*}\mathcal{O}(1)$ is positive
  because $\omega$ is a curvature form on~it. Hence by Kodaira's
  embedding theorem $L$ is ample, and in particular $X$ is an
  algebraic variety.

  \change{Furthermore if $(X,\omega)$ is Einstein then its Einstein
    constant is a positive rational number
    \cite{Hulin:sceep,Hulin:kempe}}.  Since the Ricci form $\rho$ is a
  curvature form on $-K_{X}$, this latter line bundle is positive and
  so ample. Hence in this case $X$ is a Fano variety.
\end{remark}

Let $X$ be a compact toric manifold with torus $\mathbb{T}$.  A map
$\varphi\colon X \to \mathbb{P}^{s} $ is \emph{toric} if there exists
a homomorphism of tori
$\zeta\colon \mathbb{T}\to (\mathbb{C}^{\times})^{s}$ such that
\begin{displaymath}
  \varphi(t\cdot_{X}x)   = \zeta(t)\cdot_{\mathbb{P}^{s}} \varphi(x)  \quad \text{ for all } t\in \mathbb{T} \text{ and } x\in X.
\end{displaymath}
The restriction of this map to the principal open subset is
described in terms of monomials: there are $\alpha_{j}\in \mathbb{C}$,
$j=0,\dots, s$, not all zero and $m_{j}\in M$,
$j=0,\dots, s$,~such~that
\begin{displaymath}
  \varphi(x)=(\alpha_{0}\chi^{m_{0}}(x):\cdots:\alpha_{s}\chi^{m_{s}}(x)) \quad \text{ for all } x\in X_{0}. 
\end{displaymath}
These scalars and vectors form a set of \emph{coefficients} and
\emph{exponents} for~$\varphi$.

The pullback $L=\varphi^{*} \mathcal{O}(1)$ is a line bundle on $X$
whose polytope coincides with the convex envelope of this set of
exponents \change{up to a translation by an element of $M$,~that~is}
\begin{equation}
  \label{eq:26}
\change{\Delta_{L}=  \conv(m_{0},\dots, m_{s}) +m\subset M_{\mathbb{R}} \quad \text{ with } m\in M.}
\end{equation}

\begin{definition}
  \label{def:11}
  A finite subset $S \subset M$ is \emph{unimodular} if for every
  vertex of its convex hull $v\in \conv(S)$ \change{there are exactly
    $n$ edges $E_{1},\dots, E_{n}$ of this polytope containing $v$ and
    for each $i$ there is a vector $w_{i}\in E_{i}\cap S$ such that
    $w_{1}-v,\dots, w_{n}-v$ is a basis of the lattice $M$.}
\end{definition}

We next characterize the toric maps that are full immersions. \change{Recall that $\Sigma_{X}$ denotes the fan of $X$.}

\begin{lemma}
\label{lem:14}
Let $\varphi\colon X\to \mathbb{P}^{s}$ be a toric map with
coefficients $\alpha_{j}\in \mathbb{C}$, $j=0,\dots, s$, and exponents
$m_{j}\in M$, $j=0,\dots, s$. The following conditions are
equivalent:
  \begin{enumerate}[leftmargin=*]
  \item \label{item:12} $\varphi$ is a full immersion,
  \item \label{item:13} $\alpha_{j}\ne 0$ for all~$j$,
    $m_{j}\ne m_{k}$ for all $j\ne k$, the set
    $\{m_{0},\dots, m_{s}\}$~is~unimodular, \change{and the normal
      fan of $\conv(m_{0},\dots, m_{s})$ coincides with $\Sigma_{X}$}.
  \end{enumerate}
\end{lemma}

\begin{proof}
  First assume that \eqref{item:12} holds.  Since $\varphi$ is full
  the first two conditions in \eqref{item:13} are necessary because
  otherwise the image of this map would be contained in a hyperplane.
  For the third condition, let $v=m_{j_{0}}$ be a vertex of the
  polytope $ \conv(m_{0},\dots, m_{s})$. \change{After reordering the
    exponents and rescaling the coefficients we can assume $v=m_{0}$
    and $\alpha_{0}=1$.}
  
  Since $\varphi$ is an immersion we have that $X$ is a compact smooth
  toric variety and $L=\varphi^{*}\mathcal{O}(1)$ is an ample line
  bundle on it. By \eqref{eq:26} we have that $m_{0}$ is a vertex of
  the polytope~$\Delta_{L}$, and so it corresponds to an
  $n$-dimensional cone \change{$\sigma \in \Sigma_{X}$}. Hence $\varphi$
  restricts to a monomial map between the charts
  $X_{\sigma}\simeq \mathbb{C}^{n}$ and
  $(z_{0}\ne 0)\simeq \mathbb{C}^{s}$ that writes down in these
  coordinates as
  \begin{equation*}
  \mathbb{C}^{n} \longrightarrow \mathbb{C}^{s}, \quad z\longmapsto (\alpha_{1} z^{a_{1}}, \dots, \alpha_{s}z^{a_{s}}),
\end{equation*}
where $a_{j}\in \mathbb{Z}^{n}$ corresponds to the difference
$m_{j}-m_{0}$ through the isomorphism $M \simeq \mathbb{Z}^{n}$ given
by the identification $X_{\sigma}\simeq \mathbb{C}^{n}$, and
$ z^{a_{j}}= z_{1}^{a_{j,1}}\cdots z_{n}^{a_{j,n}}$ for
$z=(z_{1},\dots, z_{n})\in \mathbb{C}^{n}$.

The Jacobian matrix of this map writes down as
\begin{equation*}
  J_{\varphi}(z)= (\alpha_{j} a_{j,k}z^{a_{j}-e_{k}})_{j,k} \in \mathbb{C}^{s\times n}
  \quad \text{ for }
\change{  z\in \mathbb{C}^{n}},
\end{equation*}
where $e_{k}$ denotes the $k$-th vector in the standard basis of
$\mathbb{Z}^{n}$. Evaluating at the origin we get
$ J_{\varphi}(0)_{j,k} \ne 0 $ if and only if $ a_{j}=e_{k}$.  Since
$\varphi$ is an immersion we have
\begin{displaymath}
\rank(J_{\varphi}(0))=n
\end{displaymath}
and so there are $ j_{1},\dots, j_{n} \in \{0,\dots, s\}$ such that
$a_{j_{k}}=e_{k}$ for each $k$. From here we get that
$\conv(m_{0},\dots, m_{s})$ has exactly $n$ edges
$E_{1}, \dots, E_{n}$ containing $v$, and for each $k$ there is
$m_{j_{k}}\in E_{k}$ such that $m_{j_{k}}-m_{0}$ corresponds to
$e_{k}$ under the above isomorphism. Hence
\begin{displaymath}
 m_{j_{1}}-m_{0}, \dots, m_{j_{n}}-m_{0} 
\end{displaymath}
is a basis of $M$, and since this holds for every vertex of the
polytope we conclude that the set $\{m_{0},\dots, m_{s}\}$ is
unimodular.

Conversely if \eqref{item:13} holds then the toric map $\varphi$ is
full: otherwise there would be a linear relation between the monomials
$\alpha_{j}\chi^{m_{j}}$, $j=0,\dots, s$, which is not possible
because they are different and nonzero.

\change{It remains to prove that $\varphi$ is an immersion.  The
  affine charts associated to the $n$-dimensional cones of the fan
  cover $X$, and so it is enough to show that $\varphi$ is an
  immersion on each of these charts. 
  
  Let $\sigma\in \Sigma_{X}^{n}$. Since $\Sigma_{X}$ coincides with
  the normal fan of~$\conv(m_{0},\dots, m_{s})$, this $n$-dimensional
  cone corresponds to a vertex $v$ of this polytope. Hence
  $v=m_{j_{0}}$ for some index  $j_{0}$, and since the set
  of exponents is unimodular there exist other indexes
  $j_{1},\dots, j_{k}$ such that the vectors
  \begin{equation}
\label{eq:3}
b_{k}=m_{j_{k}}-m_{j_{0}}, \quad k=1,\dots, n,
  \end{equation}
  give all the edge directions of the polytope at $v$ and form a basis
  of $M$.  After reordering the exponents and rescaling the
  coefficients we can assume $ v=m_{0}$, $ \alpha_{0}=1$ and
  $b_{k}=m_{k}$, $ k=1,\dots, n$.  Then we can write the restriction
  of the map $\varphi$ to the charts $X_{\sigma}\simeq \mathbb{C}^{n}$
  and $(z_{0}\ne 0)\simeq \mathbb{C}^{s}$ as
  \begin{displaymath}
  \mathbb{C}^{n} \longrightarrow \mathbb{C}^{s}, \quad z\longmapsto (\alpha_{1} z^{a_{1}}, \dots, \alpha_{s}z^{a_{s}}),
\end{displaymath}
where $a_{j}\in \mathbb{Z}^{n}$ corresponds to the difference
$m_{j}-m_{0}$ through the isomorphism $M \simeq \mathbb{Z}^{n}$ given
by the basis in \eqref{eq:3}. In our situation $a_{k}$, $k=1,\dots, n$,
coincides with the standard basis of $\mathbb{R}^{n}$, which readily
implies that $   \rank( J_{\varphi}(z))=n $ for all $z\in \mathbb{C}^{n}$
  and so $\varphi$ is an immersion on $X_{\sigma}$.}
\end{proof}

We now show that every projectively induced toric K\"ahler form on $X$
is induced by a toric full immersion. In particular this is also the
case for a projectively induced K\"ahler--Einstein form on this compact
toric manifold.

\begin{lemma}
  \label{lem:1}
  Let $\omega$ be a toric K\"ahler form on $X$ induced by a full
  immersion $\varphi\colon X\to \mathbb{P}^{s}$. Then there is
  $A\in \PU(s+1)$ such that the full immersion
  $A\, \varphi \colon X\to \mathbb{P}^{s}$ is toric.
\end{lemma}  

\begin{proof} Since $\omega$ is invariant under the action of the
  compact subtorus $\mathbb{S}$, by Calabi's rigidity theorem
  \cite[Theorem 9]{Calabi} (see also 
  \cite[Theorem~2.2]{LZ:kikmcsf}) there is a group homomorphism
  \begin{math}
\theta\colon \mathbb{S} \to \PU(s+1)
   \end{math} 
   such that
   \begin{equation}
     \label{eq:48}
 \varphi(t\cdot_{X} x) = \theta(t) \,\varphi(x) \quad \text{ for
   all } t\in \mathbb{S} \text{ and } x\in X.       
   \end{equation}
   Since the representations of $\mathbb{S}$ are diagonalizable there
   exists $A\in \PU(s+1)$ and a homomorphism
   $\zeta\colon \mathbb{S}\to (S^{1})^{s}\hookrightarrow \PU(s+1)$ such
   that
   \begin{equation}
     \label{eq:49}
    \theta(t)= A^{-1} \zeta(t)  A \quad \text{ for all } t\in \mathbb{S}.
  \end{equation}
  
  Recall that $X_{0}$ and $\mathbb{T}$ are canonically isomorphic and
  denote by $x_{0} \in X_{0}$ the point of the principal open subset
  corresponding to the unit of the torus. Setting
  $\alpha=A\, \varphi(x_{0}) \in \mathbb{P}^{s}$ we deduce
  from~\eqref{eq:48} and \eqref{eq:49} that
  \begin{displaymath}
    A \, \varphi(t \cdot_{X}x_{0})= \zeta(t) \cdot_{\mathbb{P}^{s}} \alpha  \quad \text{ for all } t\in \mathbb{S}.
  \end{displaymath}
  The homomorphism $\zeta$ is defined by a sequence of characters, and
  so it extends to a homomorphism
  $\zeta\colon \mathbb{T}\to (\mathbb{C}^{\times})^{s}$. Hence both
  $t\mapsto A \, \varphi(t\cdot_{X}x_{0}) $ and
  $t\mapsto \zeta(t) \cdot_{\mathbb{P}^{s}} \alpha $ are holomorphic
  maps that coincide on $\mathbb{S}$, and so they are equal.

Now let  $t,t'\in \mathbb{T}$ and $x=t'\cdot_{X} x_{0}\in X_{0}$. Hence 
  $t \cdot_{X} x=(t\cdot_{\mathbb{T}}t')\cdot_{X} x_{0}$ and so 
  \begin{displaymath}
    A \, \varphi(t \cdot_{X} x)= \zeta(t\cdot_{\mathbb{T}} t') \cdot_{\mathbb{P}^{s}} \alpha
    =(\zeta(t)\cdot_{(\mathbb{C}^{\times})^{s}} \zeta(t') ) \cdot_{\mathbb{P}^{s}} \alpha
    =\zeta(t)\cdot_{\mathbb{P}^{s}} (\zeta(t')  \cdot_{\mathbb{P}^{s}} \alpha) 
    =\zeta(t)\cdot_{\mathbb{P}^{s}} A\, \varphi(x),
  \end{displaymath}
  showing that the map $A\, \varphi \colon X\to \mathbb{P}^{s}$ is
  toric.
\end{proof}

We denote by $\Aut(X)$ the group of biholomorphic maps on $X$, which
is a complex Lie group \cite[Chapter~III,
Theorem~1.1]{Kobayashi:tgdg}. We also denote by $\Aut_{0}(X)$ its
identity component. It is a connected complex Lie group.  

\change{
\begin{lemma}
\label{lem:15}
Let $\omega$ be a projectively induced K\"ahler--Einstein form on $X$.
Then there exists $\psi\in\Aut_0(X)$ such that $\psi^*\omega$ is a
toric and projectively induced K\"ahler--Einstein form with the same
Einstein constant as \(\omega\). In particular $\psi^*\omega$ is
induced by a toric full immersion into a projective space.
\end{lemma}

\begin{proof}
  Since $\omega$ is K\"ahler--Einstein it is critical in the sense of
  \cite{Calabi:ekm}, and so by Theorem 3 in \emph{loc. cit.}  the
  identity component $H$ of the group of isometries of $\omega$ is a
  maximal compact subgroup of $\Aut_{0}(X)$.

  The compact subtorus~$\mathbb{S}$ is contained in $\Aut_{0}(X)$
  because it is connected. Now since any two maximal compact subgroups
  of a connected Lie group are conjugate~\cite[Chapter~XV,
  Theorem~3.1]{Hochschild:slg} this implies that there exists
  $\psi\in \Aut_{0}(X)$ such that
  \begin{displaymath}
    \psi \, \mathbb{S} \, \psi^{-1} \subset H.
  \end{displaymath}
  Thus the form $\psi^*\omega$ is toric. Furthermore pullbacks by
  holomorphic automorphisms preserve the Einstein condition, and so
  $\psi^{*} \omega$ is K\"ahler--Einstein with the same Einstein
  constant.

  Let $\varphi \colon X\to \mathbb{P}^s$ be a full immersion such that
  $\omega=\varphi^{*}\omega_{\FS}$. Hence $\psi^{*}\omega$ is induced
  by the full immersion $\varphi\circ \psi \colon X\to \mathbb{P}^s$.
  Since this form is toric, Lemma \ref{lem:1} gives an element
  $A\in \PU(s+1)$ such that
  \begin{displaymath}
   A \, \varphi \circ \psi\colon X\longrightarrow \mathbb{P}^{s}
  \end{displaymath}
 is a toric full
  immersion.  By the  invariance of the Fubini--Study form under
  $\PU(s+1)$ we have that 
\begin{math}
  \psi^{*}\omega= (A \, \varphi \circ \psi)^{*}\omega_{\FS},
\end{math}
proving the  statement.
\end{proof}
}
\subsection{Toric potentials}
\label{sec:toric-potentials-1}

Let $\varphi\colon X\to \mathbb{P}^{s}$ be a toric full immersion of a
compact toric manifold $X$ with torus $\mathbb{T}$, and let $\omega$
and $\rho$ be the associated K\"ahler and Ricci~forms.

To choose coordinates we identify the principal open subset $X_{0}$ with
$(\mathbb{C}^{\times})^{n}$.  The exponential
\begin{displaymath}
 \exp \colon \mathbb{C}^{n}/2\pi i \mathbb{Z}^{n} \rightarrow
(\mathbb{C}^{\times})^{n} 
\end{displaymath}
is an isomorphism of complex spaces, and we will study the Einstein
condition for these forms on the chart of $X$ given by this map.

By Lemma \ref{lem:14} we have 
\begin{displaymath}
     \varphi(x)=(\alpha_{0} \chi^{m_{0}}(x):\cdots:\alpha_{s}\chi^{m_{s}}(x)) \quad \text{ for } x\in X_{0}=(\mathbb{C}^{\times})^{n} 
\end{displaymath}
with $\alpha_{j}\in \mathbb{C}^{\times}$ and
$m_{j}\in M=\mathbb{Z}^{n}$ such that $m_{j}\ne m_{k}$ for all
$j\ne k$, the set $\{m_{0},\dots, m_{s}\}\subset \mathbb{Z}^{n}$ is
unimodular, \change{and the normal fan of the polytope
  $\conv(m_{0},\dots, m_{s})$ coincides with $\Sigma_{X}$.  By
  \eqref{eq:26} this polytope coincides with the polytope $\Delta_{L}$
  associated to the line bundle $L=\varphi^{*}\mathcal{O}(1)$, up to a
  translation by an element of $M$.}

We associate to the toric full immersion~$\varphi$ the Laurent
polynomial
  \begin{equation}
    \label{eq:23}
  p_{ \varphi}=\sum_{j=0}^{s} |\alpha_{j}|^{2} \chi^{m_{j}} \in \mathbb{R}_{>0}[\mathbb{Z}^{n}].
\end{equation}

Now let
$ \mathbb{C}[\mathbb{Z}^{n}] = \mathbb{C}[x_{1}^{\pm1},\dots,
x_{n}^{\pm1}] $ be the algebra of Laurent polynomials over the complex
numbers and
$\mathbb{C}(\mathbb{Z}^{n})= \mathbb{C}(x_{1},\dots, x_{n})$ its field
of rational functions.  Let $\mathbb{L}$ be any differential extension
of $\mathbb{C}[\mathbb{Z}^{n}]$ containing the logarithms of all
nonzero Laurent polynomials and consider the differential operators on
$ \mathbb{L}$ defined~as
\begin{equation}
  \label{eq:74}
  D_{i}=x_{i} \frac{\partial}{\partial} x_{i}, \quad i=1,\dots,  n.
\end{equation}

\begin{definition}[$\delta$-operator]
  \label{def:9}
For
$p\in \mathbb{C}[\mathbb{Z}^{n}] \setminus \{0\}$
we set
\begin{displaymath}
  \delta_{\mathbb{Z}^{n}}(p)= \det(D^{2}\log(p)) \in \mathbb{C}(\mathbb{Z}^{n}),
\end{displaymath}
where 
\begin{math}
D^{2}\log(p)=(D_{i}D_{j}\log(p))_{i,j} \in  \mathbb{C}(\mathbb{Z}^{n})^{n\times n} 
\end{math}
denotes the $D$-Hessian matrix of the logarithm of this Laurent
polynomial. When the lattice is clear from the context we denote this
rational function simply by $\delta(p)$.
\end{definition}

\begin{lemma}
\label{lem:16}
  Let
  $f,h\colon \mathbb{C}^{n}/2\pi i
  \mathbb{Z}^{n}\to \mathbb{R}$ be the functions respectively defined
  as
  \begin{displaymath}
    f(u+iv)=\log (p_{ \varphi}(e^{2u})) \and     h(u+iv)=\log (\delta(p_{ \varphi})(e^{2u})).
  \end{displaymath}
  Then $\omega= i \partial \overline{\partial} f$ and
  $\rho= -i \partial \overline{\partial} h$.
\end{lemma}  

\begin{proof}
  We have
  \begin{displaymath}
  f(u+iv)= \log\Big(\sum_{j=0}^{s} |\alpha_{j}
  e^{\langle m_{j},u+iv \rangle}|^{2}\Big) =(\Phi\circ \exp)^{*} \log \|\cdot\|_{2} \quad \text{ for all } u+iv \in \mathbb{C}^{n}/\change{2\pi i\,
  \mathbb{Z}^{n}}
\end{displaymath}
where $\|\cdot\|_{2}$ denotes the Euclidean norm on $\mathbb{C}^{s}$
and
$\Phi\colon (\mathbb{C}^{\times})^{n}\to \mathbb{C}^{s}\setminus
\{0\}$ the monomial map
$x\mapsto (\alpha_{0}\chi^{m_{0}}(x),\dots,
\alpha_{s}\chi^{m_{s}}(x))$. By the definition of the Fubini--Study
form (Example \ref{exm:1}) this function is a potential for $\omega$
on $\mathbb{C}^{n}/2\pi i \mathbb{Z}^{n}$.

In this chart we have
\begin{displaymath}
  \omega = i \sum_{k,l=1}^{n}  \frac{\partial^{2} f}{\partial z_{k} \partial\overline z_{l}} \, dz_{k}\wedge d\overline z_{l}
\end{displaymath}
with $z_{k}=u_{k}+i v_{k}$ for each $k$.  Setting
\begin{math}
 H =( \frac{\partial^{2} f}{\partial z_{k}  \partial\overline z_{l}})_{k,l}
\end{math}
we have that $H(u+iv)$ is a positive Hermitian matrix for all $u+iv$
and so its determinant gives a function
\begin{displaymath}
  \det(H)\colon  \mathbb{C}^{n}/2\pi i \mathbb{Z}^{n} \longrightarrow \mathbb{R}_{>0}.
\end{displaymath}
By \cite[Formula (12.6)]{Moroianu:lkg} the Ricci form can be
defined as
\begin{displaymath}
 \rho=-i \partial\overline\partial \log(\det(H )) . 
\end{displaymath}

Our potential depends only on the real part and so
$ \frac{\partial^{2} f}{\partial z_{k}\partial \overline z_{l}}=
\frac{1}{4} \frac{\partial^{2} f}{\partial u_{k}\partial u_{l}} $ for
each $k,l$.  We have
\begin{displaymath}
  \frac{\partial f}{\partial u_{k}}(u+iv)= 2\, e^{2 u_{k}} \frac{\partial \log(p_{ \varphi})}{\partial x_{k}} (e^{2 {u}}) = 2 D_{k} \log(p_{ \varphi}) (e^{2 {u}})
\end{displaymath}
and then
\begin{math}
  \frac{\partial^{2} f}{\partial u_{k}\partial u_{l}}(u+iv) = 4
  D_{k}D_{l}\log(p_{ \varphi})(e^{2 {u}}),
\end{math}
which implies
\begin{displaymath}
\frac{\partial^{2} f}{\partial z_{k}\partial \overline{z}_{l}}(u+iv)=  D_{k}D_{l}\log(p_{ \varphi})(e^{2u}).
\end{displaymath}
Thus
\begin{math}
\det(H(u+iv))=\det(  D^{2}\log(p_{ \varphi})) (e^{2u}).
\end{math}
\end{proof}

We give an algebraic characterization of the Einstein condition.

\begin{proposition}
    \label{prop:6}
    The K\"ahler form $\omega$ verifies the Einstein condition with
    constant $\lambda$ if and only if \change{there exist
      $c\in \mathbb{R}_{>0}$ and $l,m\in \mathbb{Z}^{n}$ such that
    \begin{equation*}
    \delta(\chi^{l}p_{\varphi})= c \chi^{m} (\chi^{l}p_{\varphi})^{-\lambda}.
  \end{equation*}}
\end{proposition}

\begin{proof}
  \change{In the notation of Lemma \ref{lem:16}, the Einstein
    condition for $\omega$ reduces to
  \begin{displaymath}
    h \equiv -\lambda f  \quad \pmod{\Ker( \partial\overline\partial)}
  \end{displaymath}
  and so $h+\lambda f$ is a pluriharmonic function on
  $\mathbb{C}^n/2\pi i\mathbb{Z}^n$.  Considering the universal
  covering
  \begin{math}
      \pi\colon\mathbb{C}^n\rightarrow
  \mathbb{C}^n/2\pi i \mathbb{Z}^n,
  \end{math}
  its lift to the simply connected space~$\mathbb{C}^n$ can therefore
  be written as
\begin{displaymath}
  \pi^*(h+\lambda f)=F+\overline{F} \quad \text{ on }  \mathbb{C}^n  
\end{displaymath}
for an entire function $F$.

  Since $f$ and $h$ do not depend on the variable~$v$ this is 
  also the case for $F+\overline{F}$.  Considering the representation
  of $F$ as a power series in $u+iv$ we see that this happens exactly
  when $F$ is affine, that is when
  \begin{displaymath}
  F(u+iv)=\alpha + \langle \beta,u+iv\rangle  \quad \text{ for all } u+iv \in \mathbb{C}^n
  \end{displaymath}
  with $\alpha\in \mathbb{C}$ and $\beta\in \mathbb{R}^{n}$. Setting
  $c=e^{\alpha+\overline\alpha}\in \mathbb{R}_{>0}$ we deduce
\begin{equation}\label{eq:beta}
  \delta(p_\varphi)(e^{2u})
  =c\,e^{2\langle\beta,u\rangle}
       p_\varphi(e^{2u})^{-\lambda}
\quad \text{ for all } u\in\mathbb{R}^n.
\end{equation}
The constant $\lambda$ is a positive rational number \cite{Hulin:sceep}  and we write it
as $\lambda=r/s$ in lowest terms.  Raising~\eqref{eq:beta} to the
power $s$ we have
\begin{math}
\chi^{s \beta}= c^{-1}  \delta(p_\varphi)^{s} p_{\varphi}^{r} \in \mathbb{R}(x_1,\ldots,x_n),
\end{math}
which implies that $ s\beta\in M$.

Now let $d,e\in\mathbb{Z}$ such that $dr+es=1$ and set
$ l=-d s \beta, m= e s \beta \in M$. We have
$\delta(\chi^{l}p_{\varphi})= \delta(p_{\varphi})$ because the
logarithm of a monomial is affine in logarithmic coordinates. We conclude that
\begin{displaymath}
  \delta(\chi^{l}p_{\varphi})= 
  c \chi^{\beta+\lambda l} (\chi^{l}p_{\varphi})^{-\lambda}=
  c \chi^{m} (\chi^{l}p_{\varphi})^{-\lambda}
\end{displaymath}
as stated.}
\end{proof}

\begin{example}
  \label{exm:3}
  The Laurent polynomial associated to the  identity map on $\mathbb{P}^{n}$ is
  \begin{displaymath}
    p_{\FS}=1+\sum_{i=1}^{n}x_{i}.
  \end{displaymath}
  We have $ \delta(p_{\FS})= (\prod_{i=1}^{n} x_{i} ) \,
  p_{\FS}^{-(n+1)}$ as it can be checked by  a direct
  computation or applying Lemma \ref{lem:4} below.  Using Proposition
  \ref{prop:6} this confirms that the Fubini--Study form
  $\omega_{\FS}$ is K\"ahler--Einstein with constant
  $\lambda=n+1$, in agreement with Example~\ref{exm:1}.
\end{example}

We deduce an algebraic characterization for the existence of a
K\"ahler--Einstein form on $X$ that is induced by a projective
immersion.

\begin{corollary}
  \label{cor:6}
  The following conditions are equivalent:
  \begin{enumerate}[leftmargin=*]
  \item \label{item:14} $X$ admits a projectively induced
    K\"ahler--Einstein form with constant $\lambda$,
  \item \label{item:15} there exists
    $ p \in \mathbb{R}_{>0}[\mathbb{Z}^{n}]$ with $\supp(p)$
    unimodular \change{such that the normal fan of $\NP(p)$ coincides
      with $\Sigma_{X}$ and $\delta(p)=c\chi^{m} p^{-\lambda} $ for
      some $c\in \mathbb{R}_{>0}$ and $m\in \mathbb{Z}^{n}$.}
  \end{enumerate}
\end{corollary}

\begin{proof}
  If \eqref{item:14} holds then by Lemmas \ref{lem:14} and
  \ref{lem:15} \change{we can assume without loss of generality that
    the considered K\"ahler--Einstein form is induced by a toric full
    immersion~$\varphi$ whose associated Laurent polynomial
    $p_{\varphi}$ is unimodular and the normal fan of its Newton
    polytope coincides with $\Sigma_{X}$.  Combined with Proposition
    \ref{prop:6} this gives~\eqref{item:15}.}

  Conversely assume that \eqref{item:15} holds. Write
  \begin{displaymath}
  p=\sum_{j=0}^{s} |\alpha_{j}|^{2}\chi^{m_{j}} \in \mathbb{R}_{>0}[\mathbb{Z}^{n}]
  \end{displaymath}
  with $\alpha_{j}\in \mathbb{C}^{\times}$ and
  $m_{j}\in \mathbb{Z}^{n}$, and consider the monomial map
  $ X_{0}\to \mathbb{P}^{s}$ defined as
  \begin{equation}
    \label{eq:43}
    x\longmapsto (\alpha_{0}\chi^{m_{0}}: \cdots :
  \alpha_{s}\chi^{m_{s}})
  \end{equation}
  \change{The normal fan of $\NP(p)$ coincides with  $\Sigma_{X}$,
    and in particular this lattice polytope is compatible with this
    fan. Hence} the monomials in \eqref{eq:43} extend to global
  sections of a line bundle on $X$, and so this map extends to a toric
  map $\varphi\colon X\to \mathbb{P}^{s}$. Since $\supp(p)$ is
  unimodular, by Lemma \ref{lem:14} it is a toric full immersion with
  $p_{\varphi}=p$. We conclude again with Proposition~\ref{prop:6}.
\end{proof}

\section{Algebraic potentials}
\label{sec:algebraic-potentials}

Motivated by our previous results we now place ourselves in a purely
algebraic setting and study systematically the corresponding
Monge--Ampère operator.  Throughout this section we denote by
$\mathbb{T}$ a torus of dimension $n\ge 1$, by $M$ and $N$ its
lattices of characters and co-characters, and by $M_{\mathbb{R}}$ and
$N_{\mathbb{R}}$ their vector spaces.

\subsection{The algebraic Monge--Ampère operator}
\label{sec:polyn-monge-ampere}

Let
\begin{displaymath}
 p=\sum_{j=0}^{s} c_{j} \chi^{m_{j}} \in \mathbb{C}[\mathbb{Z}^{n}] \setminus\{0\}
\end{displaymath}
with $c_{j}\ne 0$ for all $j$ and $m_{j}\ne m_{k}$ for all $j\ne k$.
With notation as in \eqref{eq:74} we have
\begin{displaymath}
D_{l}  p=\sum_{j=0}^{s} m_{j,l} c_{j} \chi^{m_{j}} \and D_{k}D_{l}p=\sum_{j=0}^{s} m_{j,k}m_{j,l} c_{j} \chi^{m_{j}} \quad \text{ for each } k, l.   
\end{displaymath}
Moreover $ D_{l}\log(p)= {D_{l}p}/{p}$ and
$D_{k}D_{l}\log(p)=(((D_{k}D_{l}p)\, p-(D_{k}p) (D_{l}p))/{p^{2}}$ and
so
\begin{displaymath}
  \delta(p)  = \frac{\det ((D_{k}D_{l}p)\, p-(D_{k}p) (D_{l}p))_{k,l}}{p^{2 n}} \in \mathbb{C}(\mathbb{Z}^{n}).
\end{displaymath}
Hence $p^{2n}$ is a denominator for the rational function $\delta(p)$.
The next lemma gives a smaller one and a formula for the corresponding
numerator.

\begin{lemma}
\label{lem:4} 
We have 
  \begin{displaymath}
    p^{n+1} \delta(p)= \sum_{J} \big( n! \vol_{\mathbb{Z}^{n}}(\conv(\{m_{j}\}_{j\in J}))\big)^{2}
    \prod_{j\in J} c_{j} \chi^{m_{j}} \in \mathbb{C}[\mathbb{Z}^{n}],
\end{displaymath}
where the sum is over the subsets $J \subset \{0,\dots, s\}$ of
cardinality $n+1$ and $\vol_{\mathbb{Z}^{n}}$ denotes the Lebesgue
measure of $\mathbb{R}^{n}$.
\end{lemma}

\begin{proof}
  For each $k,l$ we have that $(D_{k}D_{l}p)\, p-(D_{k}p) (D_{l}p)$ equals
  \begin{displaymath}    
\Big( \sum_{i=0}^{s} m_{i,k}m_{i,l} c_{i}\chi^{m_{i}}\Big) \Big( \sum_{j=0}^{s}  c_{j}\chi^{m_{j}}\Big) -
  \Big( \sum_{i=0}^{s}m_{i,k} c_{i}\chi^{m_{i}}\Big)\Big( \sum_{j=0}^{s}m_{j,l} c_{j}\chi^{m_{j }}\Big).
\end{displaymath}
Set
\begin{displaymath}
  E=
 (m_{0}\cdots m_{s}),\quad
  K=\diag( c_{0}\chi^{m_{0}},\dots, c_{s}\chi^{m_{s}})  , \quad
 w=   \sum_{j=0}^{s}m_{j} c_{j}\chi^{m_{j}}.
\end{displaymath}
Then $E$ is an $n\times (s+1)$ matrix, $K$ a diagonal
$(s+1)\times(s+1)$ matrix and $ w$ an $n$~vector that can be used to
write the $D$-Hessian matrix of $\log(p)$ as
\begin{displaymath}
p^{2} D^{2}\log(p)=p \, E  K  E^{T}-  w  \, w^{T}.
\end{displaymath}
Write $E K E^{T} = (a_{1} \cdots a_{n}) $ and
$ w=( w_{1} \cdots w_{n})^{T}$ where each $a_{j}$ is an $n$ vector and
each $ w_{j}$ a scalar. Then
\begin{displaymath}
p \,E  K  E^{T}  - w  \, w^{T}= (p a_{1}- w_{1}   w, \dots, p a_{n} -  w_{n}  w).
\end{displaymath}
By the multilinearity of the determinant function we get
\begin{displaymath}
p^{2n} \det(  D^{2}\log(p))=\det(p E  K  E^{T}) -\sum_{k=1}^{n}  w_{k}\det(p\, a_{1},\dots, p\, a_{k-1}, w,p\, a_{k+1},\dots,p\, a_{n}),
\end{displaymath}
which readily implies 
\begin{equation}
  \label{eq:68}
p^{n+1}\delta(p)=p \det(E  K  E^{T}) -\sum_{k=1}^{n}  w_{k}\det( a_{1},\dots, a_{k-1}, w, a_{k+1},\dots, a_{n}).
\end{equation}
Now consider the $(n+1)\times (s+1)$-matrix $\widetilde E=
\Big(\begin{matrix}
  E\\
  \boldsymbol{1}^{T}
\end{matrix} \Big)
=
\Big(\begin{matrix}
  m_{0} & \cdots & m_{s}\\
  1 & \cdots & 1
\end{matrix}\Big)
$.  We have
\begin{displaymath}
  \widetilde E K \widetilde E^{T}=
  \Big( \begin{matrix}
  E\\
  \boldsymbol{1}^{T}
\end{matrix}\Big)
\, K \,  ( E^{T}  \,  \boldsymbol{1})
=
\Big(\begin{matrix}
  E K E^{T}&  EK\, \boldsymbol{1} \\
 \boldsymbol{1}^{T}K E^{T} & \boldsymbol{1}^{T}K \, \boldsymbol{1}  
     \end{matrix}
     \Big)
=
\Big(\begin{matrix}
  E K E^{T}&  w\\
 w^{T} & p  
     \end{matrix}
     \Big).
\end{displaymath}
Developing the determinant of this matrix by the last row we deduce
that it coincides with the right hand side of \eqref{eq:68}. By the
Cauchy--Binet formula for the determinant of a product of rectangular
matrices we have
\begin{displaymath}
  p^{n+1}\delta(p)=\det(  \widetilde E K \widetilde E^{T})= \sum_{J}
\det((\widetilde E K)_{J}) \det (\widetilde E^{T}_{J})  =\sum_{J}
\det(\widetilde E_{J})^{2} \det( K_{J}),
\end{displaymath}
the sum being over the subsets $J\subset \{0,\dots,s\}$ of cardinality $n+1$.
Finally for each $J$ 
\begin{displaymath}
  \det(\widetilde E_{J})=\pm n! \vol_{\mathbb{Z}^{n}}(\conv(\{m_{j}\}_{j\in J})) \and
  \det( K_{J}) = \prod_{j\in J}^{n} c_{j} \chi^{m_{j}},
\end{displaymath}
which gives the formula.
\end{proof}

\begin{remark}
  \label{rem:4}
  Write $p_{J}=\sum_{j\in J}c_{j}\chi^{m_{j}}$ for each
  $J\subset \{0,\dots, s\}$ with $\# J=n+1$. Then Lemma \ref{lem:4}
  gives the decomposition
  \begin{displaymath}
    p^{n+1}\delta(p)=\sum_{J} p_{J}^{n+1}\delta(p_{J}).
  \end{displaymath}
  This lemma also shows that the $\delta$-operator is invariant with
  respect to unimodular changes of coordinates of $\mathbb{Z}^{n}$.
\end{remark}

In the $1$-dimensional case we can compute the $\delta$-operator of a
Laurent polynomial from a complete factorization of it.

\begin{lemma}
\label{lem:7}
  Let
  $p=c \, \chi^{m} \prod_{k=1}^{t} (\chi^{1}+\xi_{k})^{e_{k}} \in
  \mathbb{C}[\mathbb{Z}]\setminus \{0\}$ with $c\in \mathbb{C}^{\times}$,  $e_{k}\in \mathbb{Z}_{>0}$
  and $\xi_{k}\in \mathbb{C}^{\times}$ such that $\xi_{k} \ne \xi_{l}$
  for all $k\ne l$. Then
 \begin{displaymath}
    \delta(p)= 
    \sum_{k=1}^{t}\frac{e_{k}\xi_{k}\chi^{1}}{(\chi^{1}+\xi_{k})^{2}}.
  \end{displaymath}
\end{lemma}

\begin{proof}
  We have
  \begin{displaymath}
    \delta(p)=D^{2}\log(p)= D^{2}\log(c \chi^{m})+\sum_{k=1}^{t} e_{k} D^{2}\log(\chi^{1}+\xi_{k})= \delta(c \chi^{m})+\sum_{k=1}^{t} e_{k} \,\delta(\chi^{1}+\xi_{k}).
  \end{displaymath}
  Then we get the stated formula applying Lemma \ref{lem:4} to each of
  these terms.
\end{proof}

The $\delta$-operator also satisfies the following properties.

\begin{lemma}
  \label{lem:5}
Let $p\in \mathbb{C}[\mathbb{Z}^{n}]\setminus \{0\}$. Then
  \begin{enumerate}[leftmargin=*]
  \item \label{item:26} for all $c\in \mathbb{C}^{\times}$ and
    $m\in \mathbb{Z}^{n}$ we have
    $\delta(c \chi^{m} p)= \delta(p)$,
  \item \label{item:24} for all $\lambda\in \mathbb{Q}_{> 0}$ such that $p^{\lambda}\in  \mathbb{C}[\mathbb{Z}^{n}]$ we have 
    $\delta(p^{\lambda})=\lambda^{n}\delta(p)$,
  \item \label{item:25} for all
    $q\in \mathbb{C}[\mathbb{Z}^{l}]\setminus \{0\}$ we have
    \begin{math}
     \delta_{\mathbb{Z}^{n+l}}(p \, q)= \delta_{\mathbb{Z}^{n}}(p) \, \delta_{\mathbb{Z}^{l}}(q).
    \end{math}
  \end{enumerate}
\end{lemma}

\begin{proof}
We have 
  \begin{displaymath}
D^{2}\log(c  \chi^{m} p)=     
   D^{2}\log(c  \chi^{m}) +D^{2}\log(p) =     
   D^{2}\log(p) \and D^{2}\log(p^{\lambda})= \lambda \, D^{2}\log(p).
 \end{displaymath}
 Taking the determinants of these matrices gives the formulae in
 \eqref{item:26} and \eqref{item:24}. For~\eqref{item:25} it holds
\begin{displaymath}
  D^{2}\log(p\,q)=  D^{2}\log(p)+D^{2}\log(q)=
  \begin{pmatrix}
D^{2}\log(p) & 0\\
    0 & D^{2}\log(q)
  \end{pmatrix}
\end{displaymath}
for the $D$-Hessian matrices on $\mathbb{C}[\mathbb{Z}^{n+l}]$,
$\mathbb{C}[\mathbb{Z}^{n}]$ and $\mathbb{C}[\mathbb{Z}^{l}]$, and the
formula follows again taking determinants.
\end{proof}

\subsection{A coordinate-free setting}
\label{sec:coord-free-sett}

It will be convenient to avoid a specific choice of coordinates. Let then 
\begin{displaymath}
  p=\sum_{j=0}^{s} c_{j} \chi^{m_{j}} \in \mathbb{C}[M] \setminus\{0\}
\end{displaymath}
with $c_{j}\ne 0$ for all $j$ and $m_{j}\ne m_{k}$ for all $j\ne k$.
Set 
\begin{displaymath}
  \supp(p)=\{m_{0},\dots, m_{s}\} \subset M \and \NP(p)=\conv(\supp(p)) \subset M_{\mathbb{R}}
\end{displaymath}
for the \emph{support} and the \emph{Newton polytope} of this
Laurent~polynomial, and  consider also the associated vector space and
lattice, respectively defined as
\begin{displaymath}
M_{p,\mathbb{R}}= \sum_{j,k} \mathbb{R} (m_{j}-m_{k})
  \subset M_{\mathbb{R}} \and M_{p}= M_{p,\mathbb{R}} \cap M \subset M.
\end{displaymath}
We  have that $M_{p}$ is a subgroup
of $ M_{p,\mathbb{R}}$ with compact quotient, and we denote by
$\vol_{M_{p}}$ the Haar measure on this vector space normalized so
that
\begin{displaymath}
  \vol_{M_{p}}( M_{p,\mathbb{R}}/ M_{p})=1.
\end{displaymath}
Note that $\supp(\chi^{-m_{0}}p)\subset M_{p}$ and
$\rank(M_{p})=\dim(M_{p,\mathbb{R}}=\dim(\NP(p))$.

We introduce a polynomial variant of the $\delta$-operator that
contains its nontrivial information and  gives a meaningful
extension of this operator to the rank-deficient~case. 

\begin{definition}[$\mu$-operator]
  \label{def:5}
For $p\in \mathbb{C}[M] \setminus \{0\}$ we set
  \begin{displaymath}
  \mu(p)=p^{\rank( M_{p})+1}\delta_{M_{p}}(\chi^{-m_{0}}p) \in \mathbb{C}[M_{p}] \subset \mathbb{C}[M],
  \end{displaymath}
  where $\delta_{M_{p}}$ denotes the $\delta$-operator acting
  on the algebra $\mathbb{C}[M_{p}]$.
\end{definition}

By Remark \ref{rem:4} the $\delta$-operator does not depend on the
choice of a basis of the lattice, and by Lemma
\ref{lem:5}\eqref{item:26} it is invariant by translations of
exponents.  Hence the formula in Definition \ref{def:5} produces a
well-defined rational function, which by Lemma~\ref{lem:4} is a
Laurent polynomial.

We restate in terms of $\mu$-operator the formula given by this latter
result.

\begin{lemma}
  \label{lem:3}
Set $r=\rank(M_{p})$. Then
  \begin{displaymath}
    \mu(p)=     \sum_{J} ( r! \vol_{M_{p}}(\conv(\{m_{j}\}_{j\in J})))^{2}
    \Big(\prod_{j\in J}c_{j} \chi^{m_{j}}\Big)  ,
  \end{displaymath}
  the sum being over the subsets $J\subset \{0,\dots, s\}$ of cardinality 
  $r+1$.
\end{lemma}

We also restate in our current setting both the formula for the
$1$-dimensional case and the basic properties of the
$\delta$-operator.

\begin{lemma}
  \label{lem:12}
  Let
  $p=c \, \chi^{m} \prod_{k=1}^{t} (\chi^{1}+\xi_{k})^{e_{k}} \in
  \mathbb{C}[\mathbb{Z}]\setminus \mathbb{C}$ with $e_{k}\in \mathbb{Z}_{>0}$
  and $\xi_{k}\in \mathbb{C}^{\times}$ such that $\xi_{k} \ne \xi_{l}$
  for all $k\ne l$. Then
  \begin{displaymath}
    \mu(p)=c^{2}\chi^{2m+1} \sum_{k=1}^{t} e_{k}\xi_{k} (\chi^{1}+\xi_{k})^{2e_{k}-2} \prod_{l\ne k}(\chi^{1}+\xi_{l})^{2e_{l}}.
  \end{displaymath}
\end{lemma}

\begin{proof}
  This follows directly from Lemma \ref{lem:7} noting that here
  $ \mu(p)=p^{2} \delta(p)$.
\end{proof}
  
  \begin{lemma}
    \label{lem:6} Let $p\in \mathbb{C}[M]$ and set
    $r=\rank( M_{p})$.  Then
  \begin{enumerate}[leftmargin=*]
  \item \label{item:1} for all $c\in \mathbb{C}^{\times}$ and
    $m\in M$ we have
    $\mu(c \chi^{m} p)= c^{r+1} \chi^{(r+1)m} \mu(p)$,
  \item \label{item:2} for all $\lambda\in \mathbb{Q}_{>0}$ such that
    $p^{\lambda}\in \mathbb{C}[M]$ we have
    $\mu(p^{\lambda})=\lambda^{r} p^{(r+1)(\lambda-1)} \mu(p)$,
  \item \label{item:3} for another lattice $M'$ of rank $r'$ and
    $q\in \mathbb{C}[M']$ we have
    $\mu(p\, q)= p^{r'}q^{r}\mu(p)\, \mu(q)$.
  \end{enumerate}    
  \end{lemma}
  
\begin{proof}
  It is enough to prove these properties when $M_{p}=\mathbb{Z}^n$, in
  which case they are a direct consequence of Lemma \ref{lem:5} and
  the fact that $\mu(p)=p^{n+1}\delta(p)$. For \eqref{item:1} we have
  \begin{displaymath}
    \mu(c \chi^{m}p)= (c \chi^{m} p)^{n+1} \delta(c \chi^{m} p)= (c \chi^{m} p)^{n+1} \delta( p)
        = c^{n+1} \chi^{(n+1)m}  \mu(p).
      \end{displaymath}
      For \eqref{item:2} we have
\begin{displaymath}
  \mu(p^{\lambda})=(p^{\lambda})^{n+1} \delta(p^{\lambda})= \lambda^{n}p^{(n+1)\lambda} \delta(p)
  =\lambda^{n} p^{(n+1)(\lambda-1)} \mu(p).
\end{displaymath}
Finally for \eqref{item:3} let $q\in \mathbb{C}[\mathbb{Z}^{l}]
$. Then
 \begin{displaymath}
   \mu(p\, q)= (p\, q)^{n+l+1}\delta_{\mathbb{Z}^{n+l}}(p\, q)= (p\, q)^{n+l+1}\delta_{\mathbb{Z}^{n}}(p)\, \delta_{\mathbb{Z}^{l}}( q)=p^{l}q^{n}\mu(p)\,\mu(q),
 \end{displaymath}
 completing these verifications.
\end{proof}

\subsection{Newton polytopes}
\label{sec:newton-polytopes}

Now fix 
\begin{displaymath}
  p=\sum_{j=0}^{s}c_{j}\chi^{m_{j}}\in \mathbb{C}[M]
\end{displaymath}
such that $\supp(p)$ is a unimodular subset of $M$
(Definition~\ref{def:11}).

\begin{definition}
  \label{def:1}
For each  face $F\preceq \NP(p)$, the \emph{angle} of $\NP(p)$ at $F$ is
the cone in $N_{\mathbb{R}}$ defined as
\begin{displaymath}
  \sigma_{F}=\{u\in N_{\mathbb{R}}\mid \langle u,x-y\rangle \ge   0 \text{ for all } x\in \NP(p) \text{ and } y\in F\}.
\end{displaymath}
The collection
\begin{math}
  \Sigma_{\NP(p)}=\{\sigma_{F} \mid F \preceq \NP(p)\}
\end{math}
is a complete and unimodular fan on $N_{\mathbb{R}}$, called the
\emph{normal fan} of~$\NP(p)$.
\end{definition}

The next result gives the behavior of the Newton polytope with respect
to the $\mu$-operator.  Set $\Sigma=\Sigma_{\NP(p)}$ for short. Recall
that $\Sigma^{1}$ denotes  the set of rays of this fan and 
$u_{\tau} \in N$  the smallest nonzero lattice vector in each ray
$\tau$.

\begin{theorem}
  \label{thm:4}
Let  $ p\in  \mathbb{C}[M]$ with  $\supp(p)$  unimodular and write 
\begin{displaymath}
  \NP(p)=\{x\in M_{\mathbb{R}} \mid \langle u_{\tau} , x\rangle \ge -a_{\tau} \text{ for all } \tau\in \Sigma^{1}\}
\end{displaymath}
with \change{$a_{\tau}\in \mathbb{Z}$}. Then
  \begin{displaymath}
    \NP(\mu(p))=\{x\in M_{\mathbb{R}} \mid \langle u_{\tau} , x\rangle
    \ge -(n+1) a_{\tau} +1 \text{ for all } \tau\in
    \Sigma^{1}\}.
  \end{displaymath}
\end{theorem}

This will be derived from a characterization of the face structure and
vertex set of the Newton polytope of the Monge--Ampère polynomial of
$p$.  To state it let
\begin{displaymath}
  \NP(p)^{0} \and
 \NP(\mu(p))^{0}
\end{displaymath}
be the vertex sets of these Newton polytopes, and for each
$v\in \NP(p)^{0}$ denote by $B_{v}$ the (unique) basis of the
lattice~$M$ such that for every edge $E\preceq \NP(p)$ containing $v$
there is $b\in B_{v}$ with $v+b\in E$.

Furthermore, given a cone $\sigma \subset N_{\mathbb{R}}$ and a compact subset
$C\subset M_{\mathbb{R}}$ put
\begin{equation}
  \label{eq:44}
  C_{\sigma}=\{y\in C\mid \langle u,x-y\rangle \ge 0 \text{ for all } u\in \sigma \text{ and } x\in C \}
\end{equation}
for the subset of elements of $C$ of minimal weight in the direction
of $\sigma$.  

\begin{proposition}
  \label{prop:7}
The following properties hold: 
  \begin{enumerate}[leftmargin=*]
  \item \label{item:17}   $\Sigma_{\NP(\mu(p))}=\Sigma_{\NP(p)}$,
  \item \label{item:18}  for each 
$ v\in \NP(p)^{0}$ we have $\NP(\mu(p))_{\sigma_{v}}= \{(n+1)v + \sum_{b\in B_{v}}b\}$,
  \item \label{item:19}    $ \NP(\mu(p))^{0}= \{ (n+1)v +\sum_{b\in B_{v}}b \mid v\in \NP(p)^{0}\}.$
  \end{enumerate}
\end{proposition}

\begin{proof}
  Let $v\in \NP(p)^{0}$.  For each subset $J\subset \{0,\dots, s\}$ of
  cardinality $n+1$ whose associated exponents are affinely
  independent set
\begin{math}
  m_{J}=\sum_{j\in J} m_{j}\in M.
\end{math}
For each $j\in J$ write
$ m_{j} = v+\sum_{b\in B_{v}} \gamma_{j,b} \, b$ with
$\gamma_{j,b}\in \mathbb{Z}_{\ge 0}$ and then for each $ b\in B_{v}$
set $ \gamma_{J,b}=\sum_{j\in J} \gamma_{j,b}$.  Thus
\begin{displaymath}
m_{J}=(n+1) v + \sum_{b\in B_{v}}\gamma_{J,b}\, b.
\end{displaymath}
Since the exponents $m_{j}$, $j\in J$, are affinely independent we
have $\gamma_{J,b}\ge 1$ for all~$b$, and $\gamma_{J,b}= 1$ for
all~$b$ exactly when $J=J_{0}$ for the only index subset  such
that \change{$ \{m_{j}\}_{j\in J_{0}}=\{v\}\cup \{v+b\}_{b\in B_{v}}$}.

Let $u\in \sigma_{v}^{\circ}$ be an interior point of the
$n$-dimensional cone of $\Sigma$ corresponding to the vertex $v$.
Then $\langle u,b\rangle >0$ for all $b$ and so
  \begin{displaymath}
    \langle u,m_{J}\rangle =  (n+1) \langle u, v \rangle +  \sum_{b\in B_{v}}  \gamma_{J,b} \, \langle u , b\rangle \ge
    (n+1) \langle u, v \rangle +  \sum_{b\in B_{v}}  \langle u , b\rangle,
  \end{displaymath}
  with the equality occurring exactly when $J=J_{0}$.  Combining this
  with Lemma \ref{lem:3} we get that the minimal weight with respect
  to $u$ of the exponents of $\mu(p)$ is realized at the lattice~point
  \begin{displaymath}
w_{v}\coloneqq    m_{J_{0}}=(n+1) v+\sum_{b\in B_{v}}b.
  \end{displaymath}
Hence $\NP(\mu(p))_{\sigma_{v}}=\{w_{v}\}$, thus proving \eqref{item:18}.
  This also implies that the angle $\sigma_{w_{v}}$ of $\NP(\mu(p))$
  at the vertex $w_{v}$ contains the cone~$\sigma_{v}$. Since the
  collection $\sigma_{v}$, $v\in \NP(p)^{0}$, covers $N_{\mathbb{R}}$
  we deduce that $ \sigma_{w_{v}}= \sigma_{v}$ for all $v$, giving
  both \eqref{item:17} and \eqref{item:19}.
\end{proof}

\begin{proof}[Proof of Theorem \ref{thm:4}]
  By Proposition \ref{prop:7}\eqref{item:17} we have that
  $\NP(\mu(p))$ has the same face structure of $\NP(p)$. In
  particular their rays coincide and so by  Proposition \ref{prop:7}\eqref{item:18} 
  \begin{displaymath}
    \NP(\mu(p))=\{x\in M_{\mathbb{R}} \mid \langle u_{\tau} , x\rangle
    \ge b_{\tau}  \text{ for all } \tau\in
    \Sigma^{1}\}
  \end{displaymath}
  with $b_{\tau}=\langle u_{\tau}, (n+1) v+\sum_{b\in B_{v}}b \rangle$
  for each ray $\tau\in \Sigma^{1}$ and any vertex $v\in \NP(p)^{0}$
  lying in the facet $\NP(p)_{\tau}$. Then
  \begin{displaymath}
b_{\tau} =(n+1) \langle u_{\tau},v\rangle +\sum_{b\in B_{v}} \langle u_{\tau},b\rangle 
= -(n+1)a_{\tau} +1
\end{displaymath}
because
  \begin{math}
    \langle u_{\tau},v\rangle =-a_{\tau}
\end{math}
and  $\langle u_{\tau},b\rangle =0$ for all but one $b\in B_{v}$,
for which this weight equals  $1$.
\end{proof}

\begin{corollary}
  \label{cor:7}
  If $\NP(p)$ is a reflexive polytope then $\NP(\mu(p))= n \NP(p)$.
\end{corollary}

\begin{proof}
  This follows directly from Theorem \ref{thm:4} and the fact that
  when $\NP(p)$ is reflexive we have $a_{\tau}=1$ for all
  $\tau \in \Sigma^{1}$.
\end{proof}

\subsection{Initial parts}
\label{sec:initial-parts}

Consider again a Laurent polynomial 
\begin{displaymath}
  p=\sum_{j=0}^{s}c_{j}\chi^{m_{j}}\in \mathbb{C}[M]
\end{displaymath}
with unimodular support.  The \emph{restriction} of $p$ to a subset
$S\subset M_{\mathbb{R}}$, denoted by $p|_{S}$, is the sum of the
terms of $p$ whose exponents lie in~$S$, that is
\begin{equation}
  \label{eq:46}
  p|_{S}=\sum_{m_{j}\in S} c_{j}\chi^{m_{j}} \in \mathbb{C}[M].
\end{equation}
Given a cone $\sigma \subset N_{\mathbb{R}}$, the \emph{initial part}
of $p$ in the direction of $\sigma$ is the restriction of $p$ to the
face of its Newton polytope defined by $\sigma$ as in \eqref{eq:44},
that is
\begin{displaymath}
  \init_{\sigma}(p)=p|_{\NP(p)_{\sigma}} \in \mathbb{C}[M].
\end{displaymath}

\begin{definition}
  \label{def:8}
  Let $\Delta \subset M_{\mathbb{R}}$ be an $n$-dimensional lattice
  polytope and $F\preceq \Delta$ a facet of it. The \emph{adjacent
    polytope} of $F$ is the subset of $\Delta$ defined as
  \begin{displaymath}
F'=     \{x\in \Delta\mid \langle u_{F},x-y\rangle =1 \text{ for all } y\in F\}
  \end{displaymath}
  with $u_{F}\in N$ the primitive inner normal vector of $F$. It
  coincides with the intersection of $\Delta$ with the inner parallel
  lattice hyperplane that is closest to~$F$.
\end{definition}

We now turn to the study of the initial parts of the Monge--Ampère
polynomial of~$p$.  Recall that by Proposition
\ref{prop:7}\eqref{item:17} the Newton polytopes of $p$ and $\mu(p)$
have the same face structure, that is
\begin{displaymath}
\Sigma_{\NP(p)} =\Sigma_{\NP(\mu(p))}.
\end{displaymath}
Hence the faces of these polytopes are in dimension-reversing
bijection with the cones of this fan, that we denote again by $\Sigma$
for short.

\begin{theorem}
\label{thm:2}
Let $ p\in \mathbb{C}[M]$ with $\supp(p)$ unimodular and
$\tau\in \Sigma^{1}$. Then
\begin{displaymath}
  \init_{\tau}(  \mu(p)) = \mu(\init_{\tau}(p) )   \ p|_{F'_{\tau}}
\end{displaymath}
with $F'_{\tau}$ the adjacent polytope of the facet
$F_{\tau}\preceq \NP(p)$.
\end{theorem}

\begin{proof}
  Since $\supp(p)$ is unimodular, up to a translation and an
  isomorphism $M\simeq \mathbb{Z}^{n}$ we can assume
  \begin{displaymath}
    \NP(p)\subset (x_{n}\ge 0) \and     F_{\tau}=\NP(p)\cap (x_{n}=0).
  \end{displaymath}
  Denote by $H_{\tau} \preceq \NP(\mu(p)) $ the facet defined by the
  ray $\tau$.  By Theorem \ref{thm:4} we have
  \begin{displaymath}
\NP(\mu(p)) \subset (x_{n}\ge 1) \and
  H_{\tau}=\NP(\mu(p))\cap (x_{n}=1).    
  \end{displaymath}
  Then by Lemma \ref{lem:3}
  \begin{displaymath}
\init_{\tau}(\mu(p))=  \mu(p)|_{H_{\tau}}=  \sum_{J}
    (n!     \vol_{\mathbb{Z}^{n}}(\conv(\{m_{j}\}_{j\in J}))^{2} \prod_{j\in J} c_{j} \chi^{m_{j}}, 
  \end{displaymath}
  the sum being over the subsets $J\subset \{0,\dots, s\}$ of
  cardinality $n+1$ whose corresponding exponents are affinely
  independent and 
  \begin{displaymath}
    \sum_{j\in J} m_{j,n} =1.
  \end{displaymath}
  These are the index subsets of the form $J=I\cup \{k\} $ for any
  $I\subset \{0,\dots, s\}$ of cardinality~$n$ whose exponents
  are affinely independent and verify $m_{i,n}=0$ for all $i\in I$, and
  any $0\le k\le s$ with $m_{k,n}=1$.
  For each  decomposition  $J=I\cup\{k\}$ we have
  \begin{displaymath}
    n!    \vol_{\mathbb{Z}^{n}}(\conv(\{m_{j}\}_{j\in J}) = (n-1)!\vol_{\mathbb{Z}^{n-1}}(\conv(\{m_{i}\}_{i\in I} ))
\end{displaymath}
and $     \prod_{j\in J} c_{j}\chi^{m_{j}}= c_{k} \chi^{m_{k}} \, \prod_{i\in I} c_{i}\chi^{m_{i}}$.  Hence
  \begin{align*}
\init_{\tau}(\mu(p))&=  \sum_{I,k} ( (n-1)!     \vol_{\mathbb{Z}^{n-1}}(\conv(\{m_{i}\}_{i\in I}))^{2}
                        \, c_{k} \chi^{m_{k}} \prod_{i\in I} c_{i} \chi^{m_{i}}\\
                    &   = \Big(
    \sum_{I} ((n-1)!     \vol_{\mathbb{Z}^{n-1}}(\conv(\{m_{i}\}_{i\in I}))^{2}
                      \prod_{i\in I} c_{i} \chi^{m_{i}} \Big) \Big( \sum_{k} c_{k} \chi^{m_{k}} \Big) \\
    & =\mu(\init_{\tau}(p))  \ p_{F_{\tau}^{'}},
  \end{align*}
    which is the stated formula.   
\end{proof}

By a  descent argument we obtain a formula for an arbitrary
initial part~of~$\mu(p)$.

\begin{corollary}
\label{cor:8}
Let $ p\in \mathbb{C}[M]$ with $\supp(p)$ unimodular and
$\sigma\in \Sigma^{r}$ with $1\le r\le n$. Let $F_{\sigma}$ be the
$(n-r)$-dimensional face of $\NP(p)$ corresponding to $\sigma$ and
$G_{i}$, $i=1,\dots, r$, the faces of $\NP(p)$ of dimension $n-r+1$
containing $F_{\sigma}$.  Then
\begin{displaymath}
\init_{\sigma}  (\mu(p)) =\mu(\init_{\sigma}(p)) \, \prod_{i=1}^{r} p|_{F_{\sigma}^{(i)}}
\end{displaymath}
with $F_{\sigma}^{(i)} \subset G_{i}$ the adjacent polytope of
$F_{\sigma}$ as a facet of $ G_{i}$.
\end{corollary}

\begin{proof}
  Choose a decomposition $\sigma= \varsigma + \tau$ with
  $\varsigma \in \Sigma^{r-1}$ and $\tau \in \Sigma^{1}$, which is
  possible by the unimodularity of $\Sigma$.  By induction 
  \begin{equation}
    \label{eq:45}
    \init_{\varsigma}(\mu(p))= \mu(\init_{\varsigma}(p)) \, \prod_{i=1}^{r-1} p|_{F_{\varsigma}^{(i)}},
  \end{equation}
  where $F_{\varsigma}^{(i)}$, $i=1,\dots, r-1$, are the adjacent
  polytopes of the $(n-r+1)$-dimensional face $F_{\varsigma}$.
We also have
  \begin{displaymath}
    \init_{\tau}(    \init_{\varsigma}(\mu(p)))=\init_{\sigma}(\mu(p)), \quad   
    \init_{\tau}(    \init_{\varsigma}(p))=\init_{\sigma}(p)
  \end{displaymath}
  and $ \init_{\tau}(p|_{F_{\varsigma}^{(i)}})= p_{F_{\sigma}^{(i)}}$,
  $i=1,\dots, r-1$, for the adjacent polytopes $ F_{\sigma}^{(i)}$ of
  the face~$F_{\sigma}$ that are not contained in $F_{\varsigma}$.

  Hence taking initial parts in the direction of the ray $\tau$ in
  \eqref{eq:45} and applying the multiplicativity of initial parts and
  Theorem \ref{thm:2} we get
  \begin{displaymath}
    \init_{\sigma}(\mu(p)) =
 \init_{\tau}(  \mu(\init_{\varsigma}(p))) \, \prod_{i=1}^{r-1}     \init_{\tau}( p|_{F_{\varsigma}^{(i)}}) 
    =   \mu(\init_{\sigma}(p)) \ p|_{F_{\sigma}^{(r)}} \, \prod_{i=1}^{r-1}      p|_{F_{\sigma}^{(i)}}, 
  \end{displaymath}
  which gives the formula. 
\end{proof}

\subsection{Algebraic variants of the Einstein condition}
\label{sec:kahl-einst-cond}

Now let $X$ be a  toric Fano manifold with torus $\mathbb{T}$, and
recall that $X$ admits a projectively induced K\"ahler--Einstein form
with constant $\lambda $ if and only if there exists
$ q \in \mathbb{R}_{>0}[M]$ such that $\supp(q)$ is unimodular,
\change{the normal fan of $\NP(q)$ coincides with $\Sigma$ and}
\begin{equation}
  \label{eq:13}
 \mu(q)=c\chi^{m} q^{n+1-\lambda}  
\end{equation}
for some $c\in \mathbb{R}_{>0}$ and $m\in M$ (Corollary \ref{cor:6}).

In the study of the Einstein condition it is standard to normalize the
K\"ahler form to reduce from an arbitrary constant $\lambda>0$ to the
case ${\lambda=1}$.  For projectively induced metrics  this normalization is  more
delicate because one cannot ensure that it can be done
through a projective immersion, contrary to what might be inferred from  
\cite[page 485]{ArezzoLoiZuddas_hbm} and \cite[Section
2.5.7]{MS:TKEmicps}, see Remark \ref{rem:7} below.  Taking this issue
into account we introduce the subset of Laurent polynomials
\begin{displaymath}
\mathbb{R}[M]^{+}=\{ p\in \mathbb{R}[M] \mid p^{\nu}\in  \mathbb{R}_{>0}[M] \text{ for some } \nu\in \mathbb{Z}_{>0}\}.
\end{displaymath}

We obtain the next algebraic characterization of the Einstein
condition for a K\"ahler form that is induced by a full toric
immersion, similar to that  considered in
\cite[Lemma~4.1]{ArezzoLoiZuddas_hbm} and \cite[Remark
2.4]{MS:TKEmicps}.

\begin{proposition}
  \label{prop:1}
The
following conditions are equivalent:
  \begin{enumerate}[leftmargin=*]
  \item \label{item:11} $X$ admits a projectively
  induced K\"ahler--Einstein form,
\item \label{item:16} there exists $ p \in \mathbb{R}[M]^{+}$ with
  $\supp(p)$ unimodular, $\NP(p)= \Delta_{-K_{X}}$ and $\mu(p)=p^{n}$.
  \end{enumerate}
\end{proposition}

\begin{proof}
  Assume \eqref{item:16} and choose an integer $\nu>0$ such that
  $q\coloneqq p^{\nu}\in \mathbb{R}_{>0}[M]$. We have that $\supp(q)$
  is unimodular and $ \NP(q)=\nu \NP(p)= \nu\,
  \Delta_{-K_{X}}$. \change{Since $X$ is Fano this implies that the
    normal fan of $\NP(q)$ coincides with $\Sigma$}. By
  Lemma~\ref{lem:6}\eqref{item:2}~we~also~have
  \begin{displaymath}
\mu(q)=     \mu(p^{\nu})=\nu^{n} p^{(n+1)(\nu-1)} \mu(p) =\nu^{n} p^{(n+1)(\nu-1)+n} = \nu^{n} q^{n+1-1/\nu}.
  \end{displaymath}
By \eqref{eq:13}  $X$ admits a projectively induced  K\"ahler--Einstein form with
  constant~$\nu^{-1}$.

  Conversely assume that $X$ admits a projectively induced
  K\"ahler--Einstein form with constant $\lambda$, so that the
  conditions in \eqref{eq:13} holds.  With notation as therein~set
\begin{equation}
  \label{eq:30}
\change{p=\gamma \chi^{-m}q^{\lambda} \quad \text{ with } \gamma=(c\lambda^{n})^{-1}.}
\end{equation}
Hence $p\in \mathbb{R}[M]^{+}$ because it is a
rational function that has a positive integral power lying in
$\mathbb{R}_{>0}[M]$. Furthermore $\supp(p)$ is unimodular 
and by Lemma \ref{lem:6}(\ref{item:1},\ref{item:2})
\begin{multline*}
  \change{\mu(p)=\gamma^{n+1} \chi^{-(n+1)m}\lambda^{n} q^{(n+1)(\lambda-1)}\mu(q)}\\
  \change{= \gamma^{n+1}\chi^{-(n+1)m}\lambda^{n}
q^{(n+1)(\lambda-1)} c\,  \chi^{m} q^{n+1-\lambda}
=   (\gamma \chi^{-m}q^{\lambda})^{n}=   p^{n}.}
\end{multline*}
\change{Since the normal fan of $\NP(p)$ coincides with $\Sigma$  we can write this polytope as
\begin{displaymath}
  \NP(p)= \{x\in M_{\mathbb{R}} \mid \langle u_{\tau},x\rangle \ge -a_{\tau} \text{ for all } \tau\in \Sigma^{1}\}
\end{displaymath}
with  $a_{\tau}\in \mathbb{Z}$.
By Theorem \ref{thm:4}
\begin{displaymath}
    \NP(\mu(p))= \{x\in M_{\mathbb{R}} \mid \langle u_{\tau},x\rangle \ge -(n+1)a_{\tau}+1\}.
\end{displaymath}
Combining this with the  equation $\mu(p)=p^{n}$ we deduce that $a_{\tau}=1$ for all
$\tau$ and so $\NP(p)=\Delta_{-K_{X}}$, as stated.}
\end{proof}

\begin{remark}
  \label{rem:7}
There are Laurent polynomials
with positive coefficients that are powers of Laurent polynomials with
some negative coefficients. As an example take
 $p=2+2x- x^{2}+2x^{3}+2x^{4} \in \mathbb{R}[x^{\pm1}]$, for which
  \begin{displaymath}
    p^{2}=
    4 + 8x + 4x^3 + 17x^4 + 4x^5 + 8x^{7} + 4x^{8}\in \mathbb{R}_{>0}[x^{\pm1}].
  \end{displaymath}
  Hence in \eqref{eq:30} we cannot ensure that $q$ has positive
  coefficients, and so this Laurent polynomial is not necessarily
  given by a toric full immersion as in \eqref{eq:23}.
\end{remark}

We introduce a more flexible version of the Einstein condition in our
algebraic setting that is better suited for a recursive analysis.

\begin{definition}
\label{def:3}
Let $p\in \mathbb{C}[M]$ with $\supp(p)$ unimodular. We say that $p$
satisfies the \emph{generalized (algebraic) Einstein condition (GEC)}
if $ \mu(p) \mid p^{\kappa}$ for some $ \kappa\in \mathbb{Z}_{>0}$.
\end{definition}

As a consequence of our previous results we can show that this
condition is hereditary.  To properly state this property, for a
Laurent polynomial $p \in \mathbb{C}[M] \setminus\{0\} $ and a cone
$\sigma\in \Sigma_{\NP(p)}$ we define the \emph{translated initial
  part} of~$p$ in the direction of~$\sigma$~as
\begin{displaymath}
  p_{\sigma}=\chi^{-m} \init_{\sigma}(p) \in \mathbb{C}[M\cap \sigma^{\bot}]
\end{displaymath}
for any $m\in \supp(\init_{\sigma}(p))$, where
$\sigma^{\bot} \subset M_{\mathbb{R}}$ denotes the orthogonal
subspace~of~$\sigma$.

\begin{proposition}
\label{prop:5}
Let $p \in \mathbb{C}[M] $ with $\supp(p)$ unimodular and satisfying
GEC, and let $\sigma\in \Sigma_{\NP(p)}$. Then the translated initial
part $p_{\sigma} \in \mathbb{C}[M\cap \sigma^{\bot}]$ also verifies
that $\supp(p_{\sigma}) $ is unimodular as a subset of the lattice
$M \cap \sigma^{\bot}$ and satisfies GEC.
\end{proposition}

\begin{proof}
  It is clear that $\supp(p_{\sigma})$ is unimodular as a subset of
  $M\cap \sigma^{\bot}$ because so is the case for $\supp(p)$ as a
  subset of $M$.

  On the other hand there exists an integer $\kappa>0$ such that
  $\mu(p)$ divides $ p^{\kappa}$ because $p$ satisfies GEC.  By the
  multiplicativity of initial parts we get that
  $\init_{\sigma}(\mu(p))$ divides~$\init_{\sigma}(p)^{\kappa}$, and
  by Corollary \ref{cor:8} we have that $ \mu(\init_{\sigma}(p)) $
  divides $\init_{\sigma}(\mu(p))$.  Hence $\mu(p_{\sigma})$ divides
  $p_{\sigma}^{\kappa}$ and so $p_{\sigma}$ satisfies GEC, as stated.
\end{proof}

Finally we state the application of this property to the study of
projectively induced K\"ahler--Einstein forms on $X$.  We denote by
$\Sigma$ the fan of this toric Fano manifold.

\begin{corollary}
  \label{cor:9}
  Assume that $X$ admits a projectively induced K\"ahler--Einstein
  form, and for a cone $\sigma\in \Sigma$ denote by 
  $F_{\sigma}\preceq \Delta_{-K_{X}}$  the associated face and
  $v\in F_{\sigma}$ a lattice point in it.  Then there exists
  $q\in \mathbb{C}[M\cap \sigma^{\bot}]$ with
  $\supp(q) \subset M\cap \sigma^{\bot}$ unimodular and
  $\NP(q)=F_{\sigma}-v$ satisfying GEC.
\end{corollary}

\begin{proof}
  This follows readily from Propositions~\ref{prop:1} and \ref{prop:5}.
\end{proof}




\section{Applications}
\label{sec:applications}

In this section we apply our constructions and results to the study of
the algebraic Einstein conditions in several concrete situations. 

For simplicity here we assume that $\mathbb{T}$ is the standard torus
$(\mathbb{C}^{\times})^{n}$, so that $M=N=\mathbb{Z}^{n}$,
$M_{\mathbb{R}}=N_{\mathbb{R}}=\mathbb{R}^{n}$ and
$\mathbb{C}[M]= \mathbb{C}[x_{1}^{\pm1}, \dots, x_{n}^{\pm1}]$. To
ease the exposition we switch to the usual notation for~monomials.

\subsection{Small dimensions}
\label{sec:small-dimensions}

We first consider the $1$-dimensional case.

\begin{proposition}
\label{prop:2}
Let $p\in \mathbb{C}[x^{\pm1}]$ with $\supp(p)$ unimodular. Then $p$
satisfies GEC 
if and only if
\begin{displaymath}
    p=c \, x^{m} (x+\xi)^{\nu}  
  \end{displaymath}
  with $c,\xi\in \mathbb{C}^{\times}$, $m\in \mathbb{Z}$ and
  $\nu\in \mathbb{Z}_{> 0}$.
\end{proposition}

\begin{proof}
  Set $\NP(p)=[m,m']$ with $m\le m'$. Write
  \begin{displaymath}
    p= \sum_{j=0}^{s} c_{j}x^{m_{j}}=
    c\, x^{m} \prod_{k=1}^{t}(x+\xi_{k})^{\nu_{k}}    
  \end{displaymath}
  with $m_{j}\in \mathbb{Z}$, $\nu_{k}\in \mathbb{Z}_{>0}$ and
  $\xi_{k}\in \mathbb{C}^{\times}$ such that
  $m=m_{0}<\cdots< m_{s}=m'$ and $\xi_{k}\ne \xi_{l}$ for all
  $k\ne l$. By Lemma \ref{lem:12} we have
  \begin{displaymath}
    \mu(p)= c^{2}\, x^{2m+1} q
\prod_{k=1}^{t}(x+\xi_{k})^{2\nu_{k}-2}  \quad \text{  with }
q= \sum_{k=1}^{t} \nu_{k}\xi_{k} \prod_{l\ne k}(x+\xi_{l})^{2}.
\end{displaymath}
The leading and tail coefficients of $q$ are respectively equal to
\begin{displaymath}
  \sum_{k=1}^{t} \nu_{k}\xi_{k} = c^{-1}c_{s-1} \and   \sum_{k=1}^{t} \nu_{k}\xi_{k} \prod_{l\ne k}\xi_{l}^{2}= c^{-1}c_{1} \prod_{k=1}^{t} \xi_{k}^{2-\nu_{k}}.
 \end{displaymath}
 These quantities are nonzero because $\supp(p)$ is unimodular, and so
 $\NP(q)=[0,2t-2]$.  Since this polynomial is coprime with $p$, we
 have that $\mu(p)$ divides $ p^{\kappa}$ for a positive
 integer~$\kappa$ if and only if $q$ is a monomial. This is equivalent
 to the fact that $t=1$, which gives the statement.
\end{proof}

\begin{corollary}
  \label{cor:1}
  Let $\omega$ be a projectively induced toric K\"ahler form on
  $\mathbb{P}^{1}$ induced by a toric full immersion
  $\varphi\colon \mathbb{P}^{1}\to \mathbb{P}^{s}$ whose associated
  Laurent polynomial $p_{\varphi}\in \mathbb{R}_{>0}[x^{\pm1}]$
  satisfies GEC.  Then
  $(\mathbb{P}^{1},\omega)\simeq (\mathbb{P}^{1}, \nu \,
  \omega_{\FS})$ with $\nu\in \mathbb{Z}_{>0}$.
\end{corollary}

\begin{proof}
  By Lemma \ref{lem:14} we have that $\supp(p_{\varphi})$ is unimodular, and so 
  by Proposition~\ref{prop:2} we have that
  \begin{math}
    p_{\varphi}= cx^{m}(x+\xi)^{\nu}
  \end{math}
for an
  integer $\nu>0$ and real numbers $c, \xi>0$. Applying  Lemma
  \ref{lem:16} we get
  \begin{displaymath}
    \omega=i\, \partial \overline{\partial} \log (c \, e^{2m u}(e^{2u}+\xi)^{\nu}) = \nu \, i\, \partial \overline{\partial} \log (e^{2u}+\xi)
    = \nu A^{*}\omega_{\FS}
  \end{displaymath}
  for any linear map
  $A \colon (z_{0}:z_{1})\mapsto (\alpha z_{0}:z_{1})$ with
  $\alpha\in \mathbb{C}$ such that $|\alpha|=\xi^{1/2}$.
\end{proof}

We now focus on the $2$-dimensional case to prove a result that
severely restricts the possible Newton polygons of the bivariate
Laurent polynomials that satisfy GEC. It is a generalization of
\cite[Lemma 2.20]{MS:TKEmicps}.

Recall that for an edge $E$ of a lattice polygon
$\Delta \subset \mathbb{R}^{2}$ we denote by $E'$ its \emph{adjacent
  segment}, defined as the intersection of $\Delta$ with the inner
parallel lattice line that is closest to~$E$ (Definition
\ref{def:8}). For a lattice segment $F\subset \mathbb{R}^{2}$ we
denote by $\ell(F)$ its \emph{lattice length}, defined as the number
of lattice points in $F$ minus $1$.

Recall also that $p|_{S}$ denotes the restriction of a Laurent
polynomial $p\in \mathbb{C}[x^{\pm1}, y^{\pm1}]$ to a subset
$S\subset \mathbb{R}^{2}$ as in \eqref{eq:46}.

\begin{proposition}
  \label{prop:17}
  Let $p\in \mathbb{C}[x^{\pm1}, y^{\pm1}]$ with $\supp(p)$ unimodular
  and satisfying GEC.  Let $E$ be an edge of $\NP(p)$ and $E'$ its
  adjacent segment.
  \begin{enumerate}[leftmargin=*]
  \item \label{item:9} Let $v$ be a vertex of $\NP(p)$ and $a,b$ be
    the basis of $\mathbb{Z}^{2}$ at $v$ induced by $\supp(p)$ with
    $v+a\in E$. Then there exist $c,c'\in \mathbb{C}^{\times}$ such that
    \begin{displaymath}
      p|_{E}=c\chi^{v}(\chi^{a}+\xi)^{\ell(E)} \and      p|_{E'}=c'\chi^{v+b}(\chi^{a}+\xi)^{\ell(E')}.
    \end{displaymath}
  \item \label{item:10} Let $F$ be another edge of $\NP(p)$ and $F'$
    its adjacent segment. Then
  \begin{displaymath}
\frac{\ell(F')}{\ell(F)}=\frac{\ell(E')}{\ell(E)}. 
  \end{displaymath}
  \end{enumerate}
\end{proposition}

The figure below
illustrates the objects in this proposition.
  \begin{figure}[!htbp]
    \begin{tikzpicture}[scale=1.21]
\begin{scope}
  \filldraw[lightgray,rounded corners=0pt] (0,0) --  (-2.4,1.2) -- (-3.75-1.35,0.9-0.3) -- (-2.9-1.15,-1.15+0.25)
  -- (-1.75,-1.4)
          -- cycle;
        \end{scope}

      \draw[-] (0,0) -- (-2.4,1.2);
      \draw[-]  (-0.75,-0.6) --  (-0.75-3,-0.6+1.5);
      \draw[-] (0,0) --  (-1.75,-1.4);
      \draw[-] (-0.9,0.45) -- (-2.9,-1.15);
      \node at (0,0) [circle,fill,inner sep=1.5pt]{};
      \node at (-0.75,-0.6) [circle,fill,inner sep=1.5pt]{};
      \node at (-0.9,0.45) [circle,fill,inner sep=1.5pt]{};
      \node at (-1.65,-0.15) [circle,fill,inner sep=1.5pt]{};
      \node at (0.3,0) {$\scriptstyle  v$};
      \node at (-0.75+0.5,-0.6) {$\scriptstyle  v+a $};
      \node at (-2.4,-0.15) {$\scriptstyle  v+a+b $};
      \node at (-0.4,0.45) {$\scriptstyle  v+b $};
      \node at (-2.2,0.45) {$\scriptstyle  F' $};
\node at (-1.5,1.1) {$\scriptstyle  F $};
\node at (-2.2,-0.9) {$\scriptstyle  E' $};
\node at (-1.0,-1.1) {$\scriptstyle  E$};
\node at (-4.9,-0.5) {$\scriptstyle  \NP(p)$};
    \end{tikzpicture}
    \caption{Edges and adjacent segments}
\label{fig:1}
  \end{figure}
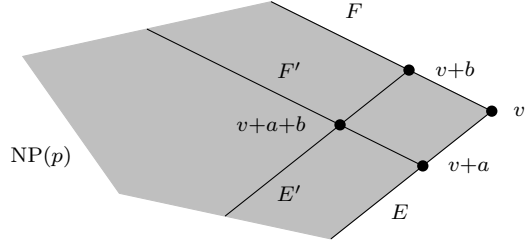

\begin{proof}
  For \eqref{item:9} let $\tau$ be the inner normal ray of $E$ so that
  this edge realizes as $E=\NP(p)_{\tau}$, and set
  $H=\NP(\mu(p))_{\tau}$ for the corresponding edge of the Newton
  polygon of~$\mu(p)$. Then Theorem~\ref{thm:2} gives the
  factorization
  \begin{equation}
    \label{eq:47}
  \mu(p)|_{H}= \mu(p|_{E}) \, p|_{E'}.
\end{equation}
By GEC we have that $\mu(p)$ divides $ p^{\kappa}$ for some integer
$\kappa>0$. By the multiplicativity of initial parts $\mu(p) |_{H}$
divides $ (p|_{E})^{\kappa}$, and so the factorization \eqref{eq:47}
implies that $\mu(p|_{E}) \, p|_{E'} $ divides $ (p|_{E})^{\kappa} $.
Hence the restriction $p|_{E}$ is a univariate Laurent polynomial that
satisfies GEC and whose Newton polytope coincides with $E$. With 
Proposition~\ref{prop:2} we get 
\begin{displaymath}
      p|_{E}=c\chi^{v}(\chi^{a}+\xi)^{\ell(E)}. 
\end{displaymath}
The corresponding expression for $p|_{E'}$ follows from the fact that
this is univariate Laurent polynomial is a factor of
$(\chi^{a}+\xi)^{\kappa}$ whose Newton polytope coincides with $E'$.

For \eqref{item:10} it is enough to consider the case when $E$ and $F$
share a vertex, as in Figure~\ref{fig:1}. \change{We have that $\ell(E')=0$ if and only if 
  $\ell(F')=0$, in which case the  equality is clear. Otherwise, by} \eqref{item:9}
we have
  \begin{displaymath}
    \begin{array}{ll} p|_{E}=c_{1} \,
\chi^{v}(\chi^{a}+\xi_{1})^{\ell(E)}, \quad &
p|_{F}=c_{2} \,
\chi^{v}(\chi^{b}+\xi_{2})^{\ell(F)}, \\
p|_{E'}=c'_{1} \,
      \chi^{v+b}(\chi^{a}+\xi_{1})^{\ell(E')}, \quad &
p|_{F'}=c'_{2} \,
\chi^{v+a}(\chi^{b}+\xi_{2})^{\ell(F')},
    \end{array}
  \end{displaymath}
  with $c_{i},c'_{i},\xi_{i} \in \mathbb{C}^{\times}$.  Examining the
  overlaps of these Laurent polynomials at the lattice points
  \change{$v, v+a, v+b,v+a+b \in \NP(p)$} we obtain the equations
  \begin{displaymath}
    \begin{array}{ll}
    c_{1}\xi_{1}^{\ell(E)}=c_{2}\xi_{2}^{\ell(F)}, \quad  &     \ell(E)\, c_{1}\xi_{1}^{\ell(E)-1}= c'_{2}\xi_{2}^{\ell(F')}, \\
 c'_{1}\xi_{1}^{\ell(E')} = \ell(F)\, c_{2} \xi_{2}^{\ell(F)-1}, \quad & \ell(E') \, c'_{1}\xi_{1}^{\ell(E')-1}
    =\ell(F')\, c'_{2}\xi_{2}^{\ell(F')-1}.
    \end{array}
  \end{displaymath}
  The statement follows by considering the ratio between the product of
  the first equation by the fourth and that of the second by the third.
\end{proof}

 Passing to the geometric setting we obtain  the next consequence.

\begin{corollary}
  \label{cor:2}
  Let $X$ be a toric surface and $\varphi \colon X\to \mathbb{P}^{s}$
  a toric full immersion such that
  $p_{\varphi} \in \mathbb{R}_{>0}[x^{\pm1},y^{\pm1}]$ satisfies
  GEC. Then the ratio $ {\ell(E')}/{\ell(E)}$ for an edge
  $E \preceq \NP(p_{\varphi})$ and its adjacent segment $E'$ does not
  depend on the choice of the edge.

  In particular, if $X$ admits a projectively induced
  K\"ahler--Einstein form then the ratio $ {\ell(E')}/{\ell(E)}$ for an
  edge $E \preceq \Delta_{-K_{X}}$ and its adjacent segment $E'$ does
  not depend on the choice of the edge.
\end{corollary}

\begin{proof}
  The first statement follows directly from Proposition
  \ref{prop:17}\eqref{item:10}. For the second, let $\omega$ be a
  projectively induced K\"ahler--Einstein form on $X$ with constant
  $\lambda$. By Lemma~\ref{lem:15} \change{we can assume that} it is
  induced by a toric full immersion
  $\varphi\colon X\to \mathbb{P}^{s}$, and by the Einstein condition
  $ \Delta_{-K_{X}}=\lambda \NP(p)$.  This statement follows then from
  the previous one.
\end{proof}

\begin{example}
  \label{exm:10}
  Let
  \begin{multline*}
    p=\alpha_{0} x ^{-1}y^{-1} +\alpha_{1}
    y ^{-1}+\alpha_{2}x y ^{-1} +\alpha_{3}x ^{2}y^{-1}
    +\alpha_{4}x ^{-1}\\ +\alpha_{5}+\alpha_{6}x +\alpha_{7}x ^{-1}y +\change{\alpha_{8}}y 
    \in \mathbb{C}[x^{\pm1},y^{\pm1}]
  \end{multline*}
  with $\alpha_{j}\ne 0$ for all $j\ne 5$. Its Newton
  polytope is the trapezoid~in~Figure~\ref{fig:2}.

  \begin{figure}[!htbp]
    \begin{tikzpicture}[scale=1.32]
      \draw[-] (-1, -1) -- (2,-1);
      \draw[-] (2, -1) -- (0,1);
      \draw[-] (0,1) -- (-1,1);
      \draw[-] (-1, 1) -- (-1,-1);
      \draw[-] (-1, 1) -- (1,-1);
      \draw[-] (-1, 0) -- (1,0);
      \node at (0,0) [circle,fill,inner sep=1.5pt]{};
      \node at (-1,-1) [circle,fill,inner sep=1.5pt]{};
      \node at (0,-1) [circle,fill,inner sep=1.5pt]{};
      \node at (1,-1) [circle,fill,inner sep=1.5pt]{};
      \node at (2,-1) [circle,fill,inner sep=1.5pt]{};
      \node at (1,0) [circle,fill,inner sep=1.5pt]{};
      \node at (0,1) [circle,fill,inner sep=1.5pt]{};
      \node at (-1,1) [circle,fill,inner sep=1.5pt]{};
      \node at (-1,0) [circle,fill,inner sep=1.5pt]{};
      \node at (0.8,0.5) {$\scriptstyle  F $};
      \node at (-0.2,0.5) {$\scriptstyle  F $};
      \node at (-0.5,-0.2) {$\scriptstyle  E' $};
      \node at (0.5,-1.3) {$\scriptstyle  E $};
    \end{tikzpicture}
    \caption{A reflexive trapezoid}
\label{fig:2}
\end{figure}
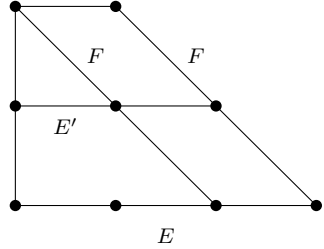

The edges
meeting at the lattice point $(2,-1)$ verify
  \begin{displaymath}
\frac{\ell(E')}{\ell(E)}=\frac{2}{3}\ne 1=\frac{\ell(F')}{\ell(F)},
  \end{displaymath}
and so   $p$ cannot satisfy GEC  by Proposition \ref{prop:17}\eqref{item:10}.
%
\end{example}

On the other hand, for the hexagon in Figure \ref{fig:3} we have
${\ell(E')}/{\ell(E)}=2$ for every edge $E$ and its adjacent segment
$E'$.
  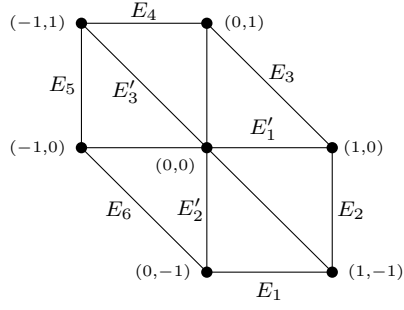
\begin{figure}[!htbp]
    \begin{tikzpicture}[scale=1.65]
      \draw[-] (0, -1) -- (1,-1);
      \draw[-] (1, 0) -- (0,1);
      \draw[-] (0,1) -- (-1,1);
      \draw[-] (-1, 1) -- (-1,0);
      \draw[-] (-1, 1) -- (1,-1);
      \draw[-] (-1, 0) -- (1,0);
      \draw[-] (-1, 0) -- (0,-1);
      \draw[-] (1, -1) -- (1,0);
      \draw[-] (0, -1) -- (0,1);
      \node at (0,0) [circle,fill,inner sep=1.5pt]{};
      \node at (-0.25,-0.15) {$\scriptscriptstyle  (0,0) $};
      \node at (0,-1) [circle,fill,inner sep=1.5pt]{};
      \node at (-0.35,-1) {$\scriptscriptstyle  (0,-1) $};
      \node at (1,-1) [circle,fill,inner sep=1.5pt]{};
      \node at (1.35,-1) {$\scriptscriptstyle  (1,-1) $};
      \node at (1,0) [circle,fill,inner sep=1.5pt]{};
      \node at (1.25,0) {$\scriptscriptstyle  (1,0) $};
      \node at (0,1) [circle,fill,inner sep=1.5pt]{};
      \node at (0.3,1) {$\scriptscriptstyle  (0,1) $};
      \node at (-1,1) [circle,fill,inner sep=1.5pt]{};
      \node at (-1.35,1) {$\scriptscriptstyle  (-1,1) $};
      \node at (-1,0) [circle,fill,inner sep=1.5pt]{};
      \node at (-1.35,0) {$\scriptscriptstyle  (-1,0) $};
      \node at (-0.13,-0.5) {$\scriptstyle  E'_{2} $};
      \node at (0.45,0.15) {$\scriptstyle  E'_{1} $};
      \node at (-0.65,0.45) {$\scriptstyle  E'_{3} $};
            \node at (0.6,0.6) {$\scriptstyle  E_{3} $};
      \node at (0.5,-1.15) {$\scriptstyle  E_{1} $};
      \node at (1.15,-0.5) {$\scriptstyle  E_{2} $};
      \node at (-0.5,1.1) {$\scriptstyle  E_{4} $};
      \node at (-1.15,0.5) {$\scriptstyle  E_{5} $};
      \node at (-0.7,-0.5) {$\scriptstyle  E_{6} $};
    \end{tikzpicture}
    \caption{A reflexive hexagon}
\label{fig:3}
\end{figure}
In spite of this, no Laurent polynomial with unimodular support and
having this Newton polytope can satisfy GEC.

\begin{proposition}
  \label{prop:3}
  Let
  \begin{displaymath}
    p=\alpha_{0}y ^{-1}+\alpha_{1}x y ^{-1}+\alpha_{2}x ^{-1}+\alpha_{3}+\alpha_{4}x +\alpha_{5}x ^{-1}y +\alpha_{6}y 
    \in \mathbb{C}[x^{\pm1}, y^{\pm1}]
  \end{displaymath}
  with $\alpha_{j}\ne 0$ for all $j$ except possibly $j=3$. Then $p$
  does not satisfy GEC.
\end{proposition}

\begin{proof}
  Let $E_{i}$, $i=1,\dots, 6$, and $E'_{j}$, $j=1,2,3$, be the edges
  and adjacent segments of the hexagon, and $p_{i}$, $i=1,\dots, 6$,
  and $p'_{j}$, $j=1,2,3$, the corresponding univariate Laurent
  polynomials. Note that $E'_{1}$ is adjacent to both $E_{1}$ and
  $E_{4}$, and a similar remark holds for $E'_{2}$ and $E'_{3}$.  By
  Proposition \ref{prop:17}\eqref{item:9}
\begin{displaymath}
\begin{array}{lll}
p_{1}=c_{1}\,y ^{-1}(x +\xi_{1}) , &
p_{2} =c_{2}\,x y ^{-1}(y +\xi_{2}) , & 
p_{3}=c_{3}\,x (x ^{-1}y +\xi_{3}) ,\\
p_{4}=c_{4}\,x ^{-1}y (x +\xi_{1}) , &
p_{5}=c_{5}\,x ^{-1}(y +\xi_{2}) , & 
p_{6}=c_{6}\, y ^{-1}(x ^{-1}y +\xi_{3}),\\
p'_{1}=c'_{1}\,x ^{-1}(x +\xi_{1})^{2}, & 
p'_{2}=c'_{2}\, y ^{-1}(y +\xi_{2})^{2}, & 
p'_{3}=c'_{3}\,x y ^{-1}(x ^{-1}y +\xi_{3})^{2},
\end{array}
\end{displaymath}
for some nonzero parameters $c_{i}$, $c'_{j}$ and $\xi_{k}$.  Hence
these parts of $p$ are encoded by $12$ parameters, and their overlapping at
each lattice point gives the following system of~$14$~equations:
\begin{displaymath}
  \begin{array}{llll}
(0,-1):  &   c_{1}\xi_{1}=c'_{2}\xi_{2}^{2}=c_{6}\,\xi_{3}, \quad &
(1,-1): &  c_{1}= c_{2}\xi_{2}=c'_{3}\xi_{3}^{2},\\
    (-1,0) :  & c'_{1}\xi_{1}^{2}=c_{5} \xi_{2}=c_{6},\quad &
  (0,0) :  & 2 c_{1}'\xi_{1}=2c'_{2}\xi_{2}=2 c'_{3}\xi_{3}, \\
    (1,0): & c'_{1}=c_{2}=c_{3}\xi_{3}, \quad & 
                                              (-1,1): & c_{4}\xi_{1}=c_{5}=c'_{3}, \\
    (0,1) : & c_{4}=c'_{2}=c_{3}.
  \end{array}
\end{displaymath}
Taking logarithms this transforms into a system of $14$ linear
equations in $12$ variables. Its solution is
\begin{displaymath}
  \begin{array}{lll}
\log(  c_{1})= -r_{1} +2 r_{2} +2 r_{3}, &  \log(c_{2})= -r_{1} + r_{2}+2r_{3}, & \log(c_{3})= r_{3}, \\
\log( c_{4})=  r_{3}, & \log(c_{5})= r_{1}, & \log(c_{6})= r_{1} + r_{2} , \\
\log(  c'_{1})= -r_{1}+r_{2}+2r_{3}, & \log(c'_{2}) = r_{3}, & \log(c'_{3})= r_{1},  \\
\log(\xi_{1}) = r_{1} - r_{3}, &  \log(\xi_{2}) = r_{2} , & \log(\xi_{3}) = -r_{1}+r_{2} + r_{3} ,
  \end{array}
\end{displaymath}
for $r_{1}, r_{2}, r_{3}\in \mathbb{C}/2\pi i \mathbb{Z}$.  Setting $\rho_{i}=e^{r_{i}}$
this implies 
\begin{align*}
  p&=p_{1}+p'_{1}+p_{4} \\
&= \rho_{1}\rho_{3}^2y^{-1} +
\rho_{1}\rho_{2}\rho_{3}xy^{-1} +
\rho_{1}\rho_{2}^{-1}\rho_{3}^2 x^{-1} +
2 \rho_{1}\rho_{3} +
\rho_{1}\rho_{2}x +  \rho_{1}\rho_{2}^{-1}\rho_{3}x^{-1}y +   \rho_{1}y .
\end{align*}
Setting furthermore $x=\rho_{2}^{-1}\rho_{3}u$, $y=\rho_{3}v$ and
dividing by $\rho_{1}\rho_{3}$, this Laurent polynomial reduces to
\begin{displaymath}
  q=v^{-1} +uv^{-1} + u^{-1} + 2  + u +  u^{-1}v  +v  =u^{-1}v^{-1}(u + v)(u + 1)(v + 1) \in \mathbb{C}[u^{\pm1},v^{\pm1}].
\end{displaymath}
The corresponding Monge--Ampère polynomial is
\begin{align*}
  \mu(q)=&
          u^2 + 2 u v + v^2 + 10 u + 2 u^2v^{-1} + 10 v + 2 u^{-1}v^2 + u^2v^{-2} + 10 u v^{-1} + 10 u^{-1}v \\
  & + u^{-2}v^2 
  + 10 u^{-1} + 2 u v^{-2}
  + 10 v^{-1} + 2 u^{-2} v + u^{-2} + v^{-2} + 2 u^{-1} v^{-1} + 18 \\
  = &u^{-2} v^{-2} (u^2 v + u v^2 + u^2 + 6 u v + v^2 + u + v) (u + v) (u + 1) (v + 1),
\end{align*}
Its first nontrivial factor is quadratic and so if it divides a power
of $p$ then it is necessarily equal to $(u+v)^{2}$, $(u+v)(u+1)$,
$(u+v)(v+1)$ or $(u+1)(v+1)$, which is is clearly not the case. Hence
$p$ does not satisfy GEC, as stated.
\end{proof}

This lattice polygon is the anticanonical polytope of the blow up of
$\mathbb{P}^{2}$ at its three fixed points \cite[Exercise
8.3.8(a)]{CoxLittleSchenck:tv}.

\begin{corollary}
  \label{cor:3}
  Let $X$ be the blowup of $\mathbb{P}^{2}$ at its three fixed points.
  Then there  is no $p\in \mathbb{C}[x^{\pm1},y^{\pm1}]$ with $\supp(p)$
  unimodular and $\NP(p)=\Delta_{-K_{X}}$ that 
  satisfies GEC. In particular $X$ does not admit a projectively
  induced K\"ahler--Einstein~form.
\end{corollary}

\subsection{Large dimensions}
\label{sec:large-dimensions} 

Finally  we apply our results to the four families of symmetric
  toric Fano manifolds discussed by Batyrev and
Selivanova in~\cite{BS:ekmstfm} and the non-symmetric examples of Nill
and Paffenholz \cite{NP:eketfmansrp}.

\begin{definition}
  \label{def:6}
  Let $X$ be a toric manifold with torus $\mathbb{T}$ and denote by
  $\Sigma$ its fan on the vector space $N_{\mathbb{R}}$. Then $X$ is
  \emph{symmetric} if the the group of automorphisms of~$\Sigma$ fixes
  only the origin. This toric manifold is \emph{centrally symmetric}
  if the map $N_{\mathbb{R}}\to N_{\mathbb{R}}$ defined as
  $u\mapsto -u$ is an automorphism of $\Sigma$.
\end{definition}

Following~\cite[Section~4]{BS:ekmstfm}, for each integer $k\ge 1$ we
denote by $V_{k}$ the $k$-th del Pezzo toric manifold introduced by
Voskresenskij and Klyachko \cite{VK:tfvrs}, see \cite[Example
4.1]{BS:ekmstfm}. It is the centrally symmetric toric Fano manifold of
dimension $n=2k$ with a fan whose cones are generated by the vectors
  \begin{equation}
    \label{eq:10}
    \pm e_{1}, \dots, \pm e_{n}, \pm (e_{1}+\dots, e_{n}), 
  \end{equation}
  where $e_{1},\dots,e_{n}$ is the standard basis of
  $N=\mathbb{Z}^{n}$.  Note that $V_{1}$ is the blowup of
  $\mathbb{P}^{2}$ at its three fixed points in Corollary \ref{cor:3}.

  Next for integers $1\le k\le m$ we denote by $S_{m,k}$ the symmetric
  toric Fano manifold of dimension $n=2m+1$ introduced by Sakane
  \cite{Sakane}, see \cite[Example 4.2]{BS:ekmstfm}.  It is the
  projectivization of the vector bundle
  $\mathcal{O}\oplus \mathcal{O}(k,-k)$ over
  $\mathbb{P}^{m}\times \mathbb{P}^{m}$, and its fan consists of the
  cones generated by the vectors
  \begin{multline}
    \label{eq:17}
    e_{1}, \dots, e_{2m}, \pm e_{2m+1},  -(e_{1}+e_{2}+\dots+ e_{m}+ke_{2m+1}), \\ -(e_{m+1}+e_{m+2}+\dots+e_{2m}-ke_{2m+1}).
  \end{multline}

  Then for integers $0\le k\le m$ we denote by $X_{m,k}$ the symmetric
  toric Fano manifold of dimension $n=2m+2$ introduced by Nakagawa
  \cite{Nakagawa}, see \cite[Example~4.3]{BS:ekmstfm}.  It is defined
  by a fan whose cones are generated by the vectors
  \begin{multline}
    \label{eq:29}
    e_{1}, \dots, e_{2m}, \pm e_{2m+1}, \pm e_{2m+2}, \pm (e_{2m+1}+e_{2m+2})    \\
    -(e_{1}+e_{2}+\dots+ e_{m}-ke_{2m+1}), -(e_{m+1}+e_{m+2}+\dots+e_{2m}+ke_{2m+1}).
  \end{multline}

  Finally for an integer $m\ge 1$ we denote by $W_{m}$ the symmetric
  toric Fano manifold of dimension $n=2m$ introduced in \cite[Example
  4.4]{BS:ekmstfm}. It is the blowup of
  $\mathbb{P}^{m}\times \mathbb{P}^{m}$ along certain $m+1$ invariant
  subvarieties of codimension $2$, and it is defined by a fan whose
  cones are generated by the vectors
  \begin{multline}
    \label{eq:32}
    e_{1}, \dots, e_{2m}, e_{1}+e_{m+1}, \dots, e_{m}+e_{2m}, \\
    -(e_{1}+\dots+ e_{m}), -(e_{m+1}+\dots+e_{2m}),
    -(e_{1}+\dots+e_{2m}).
  \end{multline}
  By \cite[Theorem 1.1]{BS:ekmstfm}, all these toric Fano manifolds
  admit~a~K\"ahler--Einstein~form.

  The next result contains  Theorem \ref{thm:1} from the introduction.
  
  \begin{theorem}
    \label{thm:3}
    Let $X$ be a symmetric toric Fano manifold of type $V_{k}$,
    $S_{m,k}$, $X_{m,k}$ or~$W_{m}$. Then for every
    $p\in \mathbb{C}[x_{1}^{\pm1}, \dots, x_{n}^{\pm1}]$ with
    $\supp(p)$ unimodular and $\NP(p)=\Delta_{-K_{X}}$ we have that
    $p$ does not satisfy GEC.  In particular $X$ does not admit a
    projectively induced K\"ahler--Einstein~form.
  \end{theorem}
  
  \begin{proof}
    Let $k\ge 1$ be an integer and denote by $\Delta_{k}$ the
    anticanonical polytope of the del Pezzo $V_{k}$.  By the
    description of this polytope in \eqref{eq:4} and of the fan of
    this toric manifold in \eqref{eq:10} we have
  \begin{displaymath}
    \Delta_{k}=\{ x\in \mathbb{R}^{n} \mid -1\le x_{1}, \dots, x_{n}, x_{1}+\dots+x_{n} \le 1\}.
  \end{displaymath}
  The subset
  \begin{displaymath}
    F=\{x\in \Delta_{k} \mid x_{i}=(-1)^{i} \text{ for } i=3,\dots, n\}
  \end{displaymath}
  is a face of $\Delta_{k}$ because it is given by the equality in
  some of the inequalities defining this polytope.  This face is an
  hexagon in the $(1,2)$-plane of $\mathbb{R}^{n}$ identifying with
  that in Proposition \ref{prop:3}, namely
  \begin{displaymath}
    F\simeq    \{(x_{1},x_{2})\in \mathbb{R}^{2}\mid -1\le x_{1},x_{2}, x_{1}+x_{2}\le 1\}.
  \end{displaymath}
  By Proposition \ref{prop:5}, if the Laurent polynomial $p $
  satisfies GEC then this is also the case for its restriction
  $p|_{F}$, which is excluded by Proposition \ref{prop:3}.
  
  Now for integers $1\le k\le m$ we denote by $\Delta_{m,k}$ the
  anticanonical polytope of the Sakane toric manifold $S_{m,k}$. By
  \eqref{eq:17} we have
  \begin{multline*}
    \Delta_{m,k}=\{(x_{1},\dots, x_{m}, y_{1},\dots, y_{m},z)\in \mathbb{R}^{n} \mid
    x_{i}, y_{j}\ge -1 \text{ for } i,j=1,\dots, m, \\ -1\le z\le 1, \ x_{1}+x_{2}+\dots+x_{m}\le 1-kz, \ y_{1}+y_{2}+\dots+y_{m}\le 1+kz\}
  \end{multline*}
  The subset of this polytope defined by the equations
  \begin{displaymath}
  x_{1}=\dots= x_{m}=y_{1}=\dots=y_{m-1}=-1  
  \end{displaymath}
  is a  face $F$ that identifies with a trapezoid in
  the $(2m,2m+1)$-plane of $\mathbb{R}^{n}$, namely
  \begin{displaymath}
    F\simeq\{(y_{m},z)   \in \mathbb{R}^{2} \mid    -1\le y_{m}\le m+kz, -1\le z\le 1\}.
  \end{displaymath}
  As in the previous case, if $p $ satisfies GEC then this is also
  the case for $p|_{F}$. Now the ratios of the lattice length of the
  two edges of this trapezoid meeting at the lattice point $(-1 ,1)$
  and its respective adjacent segments are $(m+k+1)/(m+1)$ and $1$,
  and so by Proposition \ref{prop:17}\eqref{item:10} the condition
  GEC for $p|_{F}$ is not possible.

  Next for integers $0\le k\le m$ we denote by $\Delta_{m,k}$ the
  anticanonical polytope of the Nakagawa toric manifold $X_{m,k}$. By
  \eqref{eq:29} we have
  \begin{multline*}
    \Delta_{m,k}=\{(x_{1},\dots, x_{m}, y_{1},\dots, y_{m},z,w)\in \mathbb{R}^{n}  \\ \mid
    x_{i}, y_{j}\ge -1 \text{ for } i,j=1,\dots, m,  \ -1\le z,w, z+w\le 1, \\
    x_{1}+x_{2}+\dots+x_{m}\le 1+kz, \ y_{1}+y_{2}+\dots+y_{m}\le 1-kz\}.
  \end{multline*}
  The subset defined by the equations
  \begin{displaymath}
    x_{1}=x_{2}=\dots, x_{m}=y_{1}=y_{2}=\dots=y_{m}=-1
  \end{displaymath}
  is a face $F$ that identifies with an hexagon in the
  $(2m+1,2m+2)$-plane of $\mathbb{R}^{n}$:
  \begin{displaymath}
F\simeq\{(z,w)  \in \mathbb{R}^{2}\mid     -1\le z,w, z+w\le 1\}.
  \end{displaymath}
  It is the same that appears for the del Pezzo $V_{k}$'s, and as
  therein we deduce that $p$ cannot satisfy GEC.

  Finally for an integer $m\ge 1$ we let $\Delta_{m}$ be the
  anticanonical polytope of the Batyrev-Selivanova toric manifold
  $W_{m}$. By \eqref{eq:32} we have
  \begin{multline*}
    \Delta_{m}=\{(x_{1},\dots, x_{m}, y_{1},\dots, y_{m})\in \mathbb{R}^{n} \mid
    x_{i}, y_{i}, x_{i}+y_{i}\ge -1 \text{ for } i=1,\dots, m, \\  x_{1}+\dots+x_{m},  \ y_{1}+\dots+y_{m}, \ x_{1}+\dots+x_{m}+  y_{1}+\dots+y_{m}\le 1\}.
  \end{multline*}
  The face  of this polytope defined by the equations
  \begin{displaymath}
    x_{1}=\dots =x_{m-1}=-1, \quad x_{1}+y_{1}=\dots=x_{m-1}+y_{m-1}=-1
  \end{displaymath}
  is an hexagon in the $(m,2m)$-plane:
  \begin{displaymath}
F\simeq\{(x_{m},y_{m})   \in \mathbb{R}^{2} \mid    x_{m}, y_{m}, x_{m}+y_{m}\ge -1, \ x_{m}, x_{m}+y_{m}\le m, \ y_{m}\le 1\},
\end{displaymath}
When $m=1$ this is the hexagon considered in  previous cases.
When $m\ge 2$ the ratios of the lattice length of the two edges
meeting at the lattice point $(-1,1)$ and their adjacent segments are
\begin{displaymath}
  \frac{m+1}{m} \ne  2.
\end{displaymath}
Applying again Proposition \ref{prop:17}\eqref{item:10} we get that
$p$ can neither satisfy GEC in this~case.

The last statement is a direct consequence of Corollary \ref{cor:9}.
\end{proof}

\begin{proof}[Proof of Corollary \ref{cor:4}]
  Let $X$ be a centrally symmetric toric Fano manifold.  By
  \cite[Theorem 6]{VK:tfvrs} it is isomorphic to a product of projective
  lines and del Pezzo toric manifolds.

  If $X$ is a product of projective lines then its Segre embedding
  gives a projectively induced K\"ahler--Einstein form on it.  On the
  other hand, if $X$ has a del Pezzo factor~$V_{k}$ then the
  anticanonical polytope of $V_{k}$ is a face of the anticanonical
  polytope of $X$.  Theorem \ref{thm:3} and Corollary \ref{cor:9} then
  imply that $X$ does not admit a projectively induced K\"ahler--Einstein
  form.
\end{proof}

In \cite{NP:eketfmansrp} Nill and Paffenholz presented two
non-symmetric  toric Fano manifolds in dimensions $7$ and $8$ admitting
a K\"ahler--Einstein form.  The first is the toric manifold~$X_{1}$
associated to a fan whose cones are generated by the vectors
\begin{equation}
  \label{eq:33}
  e_{1},e_{2}, e_{3},e_{4},e_{5},e_{6}, \pm e_{7}, -e_{1}-e_{7},-e_{2}-e_{7},-e_{3}-e_{7}, -e_{4}-e_{5}-e_{6}+2e_{7} \in \mathbb{R}^{7},
\end{equation}
whereas the second is the toric manifold $X_{2}$ associated to a fan whose
cones are generated by the vectors
\begin{multline}
  \label{eq:34}
  e_{1},e_{2}, e_{3},e_{4},e_{5},e_{6}, \pm e_{7}, \pm e_{8},  \pm (e_{7}-e_{8}),  \\
  -e_{1}-e_{8}, -e_{2}-e_{8},-e_{3}-e_{8},
  -e_{4}-e_{5}-e_{6}+2e_{8}\in \mathbb{R}^{8} .
\end{multline}

\begin{theorem}
  \label{thm:5}
  The toric Fano manifolds $X_{1}$ and $X_{2}$ do not admit a
  projectively induced K\"ahler--Einstein form.
\end{theorem}

\begin{proof}
  Denote by $\Delta_{i}$ the anticanonical polytope of $X_{i}$,
  $i=1,2$. By \eqref{eq:33} we have
  \begin{multline*}
    \Delta_{1}=\{(x_{1},\dots, x_{7})\in \mathbb{R}^{7}\mid -1\le x_{1},\dots, x_{6} , \ -1\le x_{7}\le 1, \\
    x_{1}+x_{7}, x_{2}+x_{7}, x_{3}+x_{7}, x_{4}+x_{5}+x_{6}-2x_{7}\le 1\}.
  \end{multline*}
  The subset of this polytope defined by the equations
  $x_{2}=\cdots= x_{6}=-1$ is a face in the $(1,7)$-plane
  identifying with a trapezoid:
  \begin{displaymath}
    F\simeq \{(x_{1},x_{7})\in \mathbb{R}^{2}\mid -1 \le x_{1},\ -1 \le x_{7} \le 1,  \  x_{1}+x_{7}\le 1\}.
  \end{displaymath}
  Let $p\in \mathbb{C}[\mathbb{Z}^{7}]$ with $\supp(p)$ unimodular. By
  Proposition \ref{prop:17} we have that $p|_{F}$ does not satisfy
  GEC, and so $p$ does not satisfy GEC by the hereditary character of
  this condition (Proposition~\ref{prop:5}). Hence by Proposition
  \ref{prop:1} we have that $X_1$ does not admit a projectively
  induced K\"ahler--Einstein form.

  Similarly by \eqref{eq:34} we have
  \begin{multline*}
    \Delta_{2}=\{(x_{1},\dots, x_{8})\in \mathbb{R}^{8}\mid -1 \le x_{1},\dots, x_{6}, \
    -1\le x_{7},x_{8}, x_{7}-x_{8}\le 1, \\ x_{1}+x_{8}, x_{2}+x_{8}, x_{3}+x_{8}, x_{4}+x_{5}+x_{6}-2x_{8}\le 1 \}.
  \end{multline*}
  The subset  defined by the equations
  $x_{1}=\cdots=x_{6}=-1$ is a face in the $(7,8)$-plane identifying
  with an hexagon:
  \begin{displaymath}
    F\simeq \{(x_{7},x_{8})\in \mathbb{R}^{2}\mid -1 \le x_{7},x_{8}, x_{7}-x_{8}\le 1\}.
  \end{displaymath}
As before, we deduce that $X_{2}$ does not admit a projectively
  induced K\"ahler--Einstein form.
\end{proof}

\begin{remark}
  \label{rem:1}
  As is clear from its proof, it is possible to state a GEC version of
  this result similar to that in Theorem \ref{thm:3}.
\end{remark}

These examples were generalized by Nakagawa \cite{Nakagawa:enp}, and
it would be interesting to see if the previous analysis extends to
this series of non-symmetric toric Fano manifolds.

\subsection*{Acknowledgments}
\change{We thank the referee for their careful reading of the
  manuscript and for pointing out several mistakes. Their insightful
  comments and examples helped us correct these errors and improve the
  exposition.}
  
Part of this work was done during a research stay at the Centre
International de Rencontres Math\'ematiques (CIRM) at Luminy and
visits to the Politecnico di Torino and the Universitat de
Barcelona. We thank these institutions for their hospitality.

Antonio Di Scala is a member of the research group in cryptography and
number theory (CrypTO) at Politecnico di Torino and GNSAGA at
INdAM. He was also partially supported by the program ``Finanzamento alla
Ricerca'' at Politecnico di Torino.  Mart\'{\i}n Sombra was partially
supported by the MICINN research project PID2023-147642NB-I00 and the
AEI project CEX2020-001084-M of the Mar\'{\i}a de Maeztu program for
centers and units of excellence in R\&D.

\bibliographystyle{amsalpha}
\bibliography{./bibliography.bib}

\end{document}